\documentclass[12pt,reqno]{amsart}

\usepackage{amssymb,bbm,esint,units,url}
\usepackage{booktabs}
\usepackage[bookmarks=false,unicode,colorlinks,urlcolor=red,citecolor=red,linkcolor=blue]{hyperref}
\usepackage{aliascnt}
\usepackage{doi}
\usepackage{tikz}
\usepackage{pgfplots} 

\usepackage{color}

\usepackage[margin=1.25in]{geometry}

\newtheorem{theorem}{Theorem}[section]

\newaliascnt{lemma}{theorem}
\newtheorem{lemma}[lemma]{Lemma}
\aliascntresetthe{lemma}

\newaliascnt{proposition}{theorem}
\newtheorem{proposition}[proposition]{Proposition}
\aliascntresetthe{proposition}

\newaliascnt{corollary}{theorem}
\newtheorem{corollary}[corollary]{Corollary}
\aliascntresetthe{corollary}

\newaliascnt{conjecture}{theorem}

\aliascntresetthe{conjecture}

\newtheorem{conjx}{Conjecture}

\newaliascnt{questx}{conjx}

\aliascntresetthe{questx}

\theoremstyle{definition}

\newaliascnt{example}{theorem}
\newtheorem{example}[example]{Example}
\aliascntresetthe{example}

\makeatletter
\def\tagform@#1{\maketag@@@{\ignorespaces#1\unskip\@@italiccorr}}
\let\orgtheequation\theequation
\def\theequation{(\orgtheequation)}
\makeatother
\def\equationautorefname~{}

\newtheorem*{remark}{Remark}

\newcommand{\arxiv}[1]{%
 \href{http://front.math.ucdavis.edu/#1}{ArXiv:#1}}

\newcommand{\h}{H}
\newcommand{\go}{G_1}
\newcommand{\gt}{G_2}
\newcommand{\ho}{H_1}
\newcommand{\htwo}{H_2}
\newcommand{\B}{{\mathbb B}}

\newcommand{\e}{\varepsilon}

\newcommand{\R}{{\mathbb R}}

\newcommand{\Rn}{{\mathbb R}^n}

\newcommand{\sinc}{\operatorname{sinc}}
\newcommand{\sech}{\operatorname{sech}}
\newcommand{\csch}{\operatorname{csch}}
\newcommand{\sinch}{\operatorname{sinch}}
\newcommand{\sign}{\operatorname{sign}}

\newcommand{\Krejcirik}{{Krej\v{c}i\v{r}\'{\i}k}}

\begin{document}

\title[Robin Laplacian --- conjectures and rectangles]{The Robin Laplacian --- spectral conjectures, rectangular theorems}
\author[]{Richard S. Laugesen}
\address{Department of Mathematics, University of Illinois, Urbana,
IL 61801, U.S.A.}
\email{Laugesen\@@illinois.edu}
\date{\today}

\begin{abstract}
The first two eigenvalues of the Robin Laplacian are investigated along with their gap and ratio. Conjectures by various authors for arbitrary domains are supported here by new results for rectangular boxes. 

Conjectures with fixed Robin parameter include: a strengthened Rayleigh--Bossel inequality for the first eigenvalue of a convex domain under area normalization; a Szeg\H{o}-type upper bound on the second eigenvalue of a convex domain; the gap conjecture saying the line segment minimizes the spectral gap under diameter normalization; and the Robin--PPW conjecture on maximality of the spectral ratio for the ball. Questions for a varying Robin parameter include monotonicity of the spectral gap and the spectral ratio, as well as  concavity of the second eigenvalue. 

Results for rectangular domains include that: the square minimizes the first eigenvalue among rectangles under area normalization, when the Robin parameter $\alpha \in \R$ is scaled by perimeter; that the square maximizes the second eigenvalue for a sharp range of $\alpha$-values; that the line segment minimizes the Robin spectral gap under diameter normalization for each $\alpha \in \R$; and the square maximizes the spectral ratio among rectangles when $\alpha>0$. Further, the spectral gap of each rectangle is shown to be an increasing function of the Robin parameter, and the second eigenvalue is concave with respect to $\alpha$. 

Lastly, the shape of a Robin rectangle can be heard from just its first two frequencies, except in the Neumann case.  
\end{abstract}

\maketitle

\section{\bf Introduction}
\label{intro}

New shape optimization conjectures are developed and old ones revisited for the first two eigenvalues of the Robin Laplacian. Along the way, conjectures are supported with theorems on the special case of rectangular domains. 

Shape optimization problems for the spectrum of the Robin Laplacian 
\begin{equation*}\label{robinproblem}
\begin{split}
- \Delta u & = \lambda u \ \quad \text{in $\Omega$,} \\
\frac{\partial u}{\partial\nu} + \alpha u & = 0 \qquad \text{on $\partial \Omega$,} 
\end{split}
\end{equation*}
have resolutely resisted techniques employed on the Neumann and Dirichlet endpoint cases ($\alpha=0$ and $\alpha=\infty$ respectively). For example, Rayleigh's conjecture that the ball minimizes the first eigenvalue among all domains of given volume was proved for Dirichlet boundary conditions by Faber and Krahn  in the 1920s, using rearrangement methods. The Neumann case  is trivial since the first eigenvalue is zero for every domain. Yet the Robin case of the conjecture, which lies between the Neumann and Dirichlet ones, was established only in the 1980s in the plane by Bossel \cite{B86}. Her extremal length methods were extended to higher dimensions by Daners \cite{D06} in 2006, followed in 2010 by a new shape optimization approach of Bucur and Giacomini \cite{BG10}. 

Lurking beyond the Neumann case lie the negative Robin parameters, for which Bareket \cite{B77} conjectured the ball might maximize the first eigenvalue among domains of given volume. Freitas and \Krejcirik\ \cite{FK15} disproved this conjecture in general with an annular counterexample, but they succeeded in proving it in $2$ dimensions when the negative Robin parameter is sufficiently close to $0$. For the second eigenvalue with negative Robin parameter, recent papers by Freitas and Laugesen \cite{FL18a,FL18b} generalize to a natural range of  parameter values the sharp Neumann upper bounds of Szeg\H{o} \cite{S54} and Weinberger \cite{W56}, with the ball being the maximizer.  

\subsection*{Overview of results} Rectangles are everyone's first choice when seeking computable examples. The Neumann and Dirichlet spectra of rectangles are completely explicit, but the Robin eigenvalues must be determined from transcendental equations (as collected in \autoref{identifyinginterval} and \autoref{identifying}), and thus are more complicated to extremize. Both positive and negative Robin parameters will be considered. Negative Robin parameters correspond in the heat equation to non-physical boundary conditions, with ``heat flowing from cold to hot''. Negative parameters do arise in a physically sensible way in a model for surface superconductivity \cite{GS07}. In any case, from a mathematical perspective the negative parameter regime is a natural continuation of the positive parameter situation. 

A \textbf{rectangular box} in $\Rn$ is the Cartesian product of $n$ open intervals. The edges can be taken parallel to the coordinate axes, by rotational invariance of the Laplacian. A \textbf{cube} is a box whose edges all have the same length. In $2$ dimensions the box is a rectangle, and the cube is a square. 

For rectangular boxes of given volume, \autoref{firsttwo} illustrates the following six results, two of which concern 
\begin{figure}
\begin{center}
\includegraphics[scale=0.45]{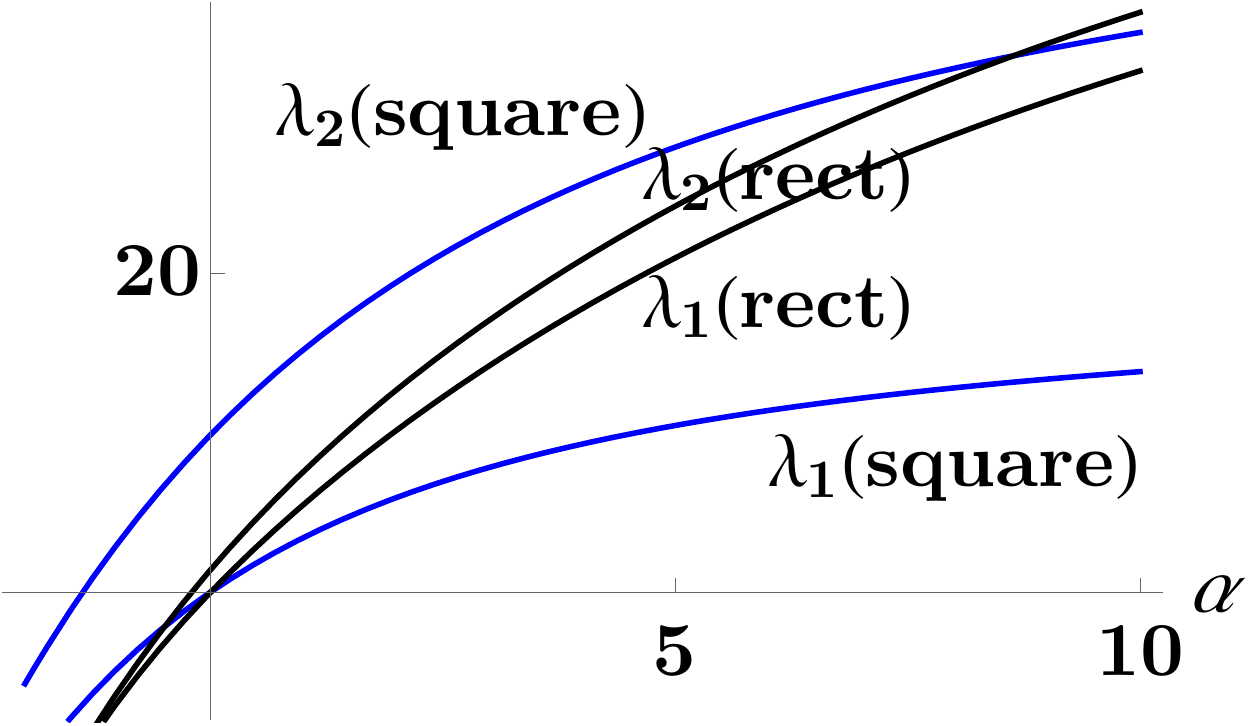}
\hspace{1.5cm}
\includegraphics[scale=0.45]{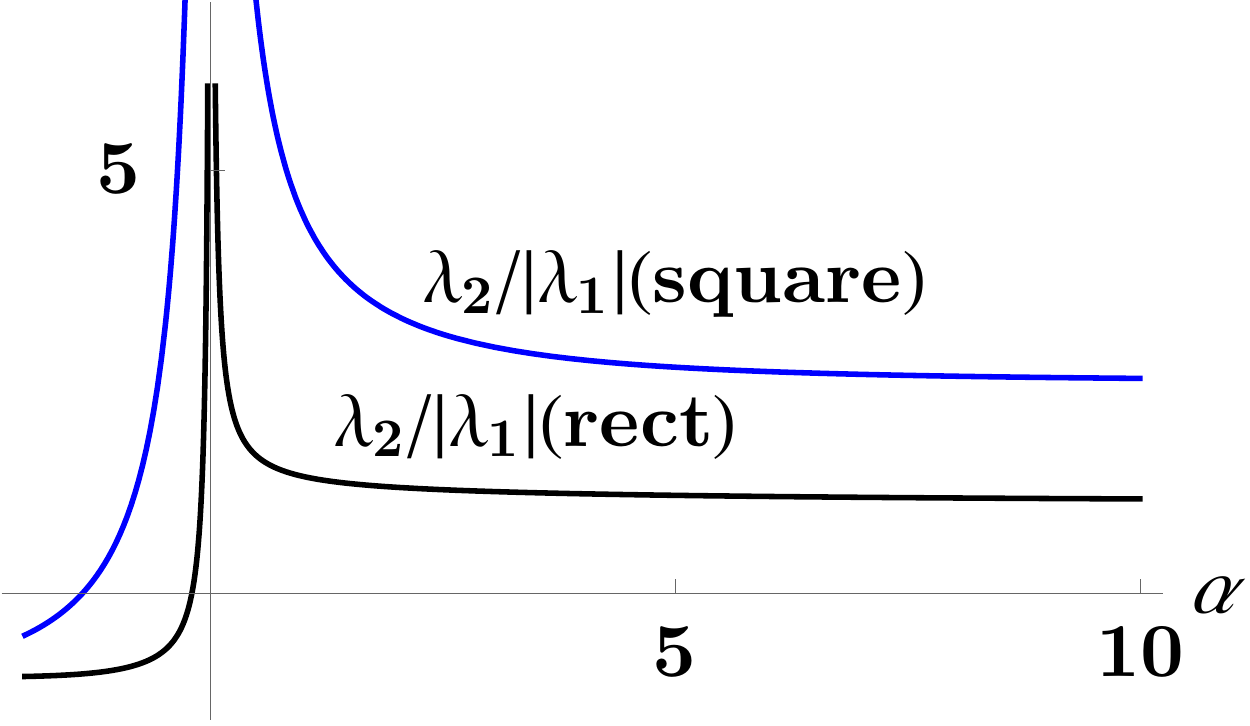}
\end{center}
\caption{\label{firsttwo} Left: the first two eigenvalues $\lambda_1$ and $\lambda_2$ for the unit square and for a rectangle with area $1$ and aspect ratio $7$, plotted as functions of the Robin parameter $\alpha$. Right: the ratio $\lambda_2/|\lambda_1|$ for the square and rectangle.}
\end{figure}
dependence on the Robin parameter while the other four involve shape optimization: 
\begin{itemize}
\item monotonicity of the spectral gap $\lambda_2-\lambda_1$ as a function of $\alpha \in \R$ (\autoref{gapmonot}; due to Smits \cite[Section 4]{Sm96} for $\alpha>0$) 
\item concavity of the first and second eigenvalues with respect to $\alpha \in \R$ (\autoref{secondeigconcave}) 
\item maximality of the cube for the first eigenvalue when $\alpha<0$, and minimality of the cube when $\alpha>0$ (\autoref{lambda1higherdim}; minimality when $\alpha>0$ is due to Freitas and Kennedy \cite[Theorem 4.1]{FK18} in $2$ dimensions and to Keady and Wiwatanapataphee \cite{KW18} in all dimensions), 
\item maximality of the cube for the second eigenvalue when $\alpha \leq 0$ (\autoref{lambda2higherdim}); when $\alpha>0$ this maximality can fail, as seen on the left of \autoref{firsttwo},
\item maximality of the cube for the magnitude of the $\lambda_2$-horizontal intercept, that is, for the first nonzero Steklov eigenvalue (\autoref{sigma1higherdim}; this was proved in a stronger form with a different approach by Girouard \emph{et al.}\ \cite{GLPS17}) 
\item maximality of the cube for the spectral ratio $\lambda_2/|\lambda_1|$ (\autoref{ratio}).
\end{itemize}

Now let us place these rectangular results in context with conjectures and results for general domains. The first result above proves a special case of Smits' monotonicity conjecture for the spectral gap on arbitrary convex domains \cite[Section 4]{Sm96}; see \autoref{NeumannDirichletGapConj}. The second result suggests concavity of the second eigenvalue (\autoref{lambda2concavity}) when the Robin parameter is positive and the domain is convex. The third result is of Bareket/Rayleigh type. When $\alpha>0$ it is the rectangular analogue of the Bossel--Daners theorem for general domains. The fourth result, about maximizing the second eigenvalue, is the rectangular version of Freitas and Laugesen's result \cite{FL18a} for general domains with $\alpha \in [-(1+1/n)R^{-1},0]$, where $R$ is the radius of the ball having the same volume as the domain. That $\alpha$-range for general domains is not thought to be optimal. The fifth result is of Brock-type for the Steklov eigenvalue. The sixth one, about maximality of the spectral ratio, motivates \autoref{ratioconj} later in the paper for general domains. 

Further, the spectral gap of a rectangular box is shown in \autoref{gapD} to be minimal for the degenerate rectangle of the same diameter, for each $\alpha \in \R$, which is consistent with \autoref{robingap} later for arbitrary convex domains when $\alpha>0$. 

\smallskip
The most difficult results in the paper arise when the Robin parameter is scaled by the perimeter $L$ of a planar domain, that is, when the Robin parameter is $\alpha/L$. For rectangles with given area, \autoref{firsttwoperim} and its close-up in \autoref{firsttwoperimcloseup} illustrate: 
\begin{figure}
\begin{center}
\includegraphics[scale=0.45]{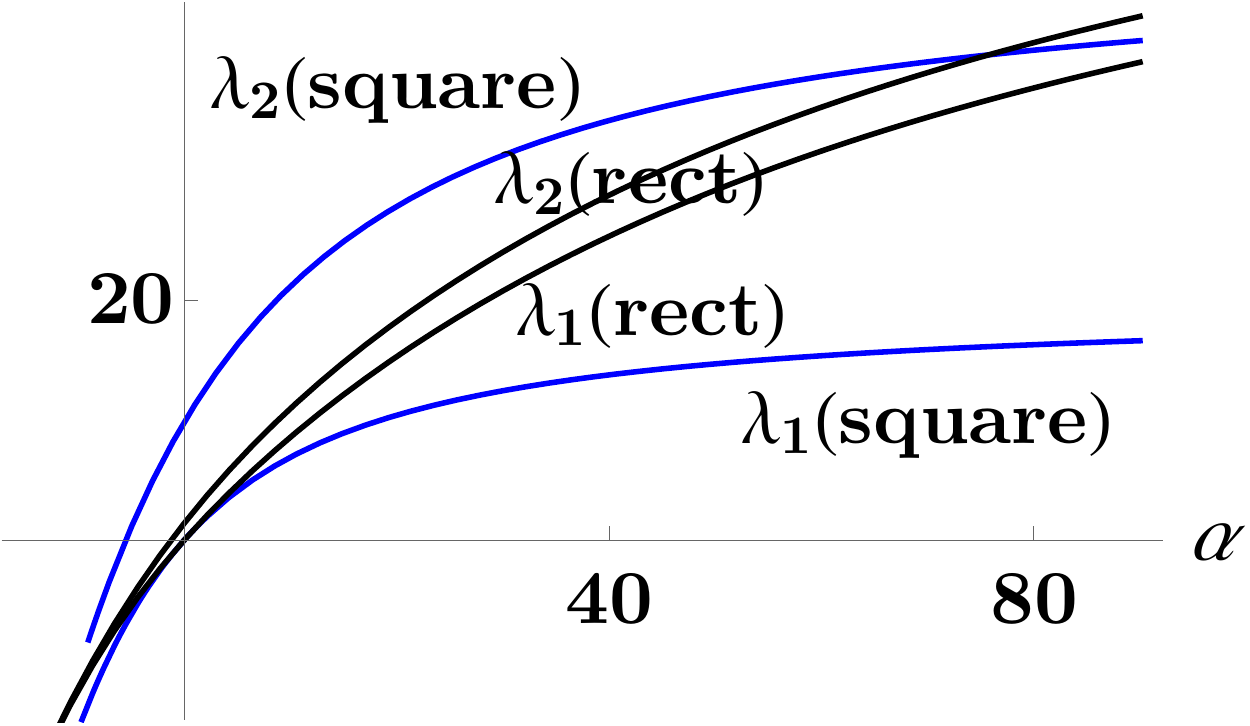}
\hspace{1.5cm}
\includegraphics[scale=0.45]{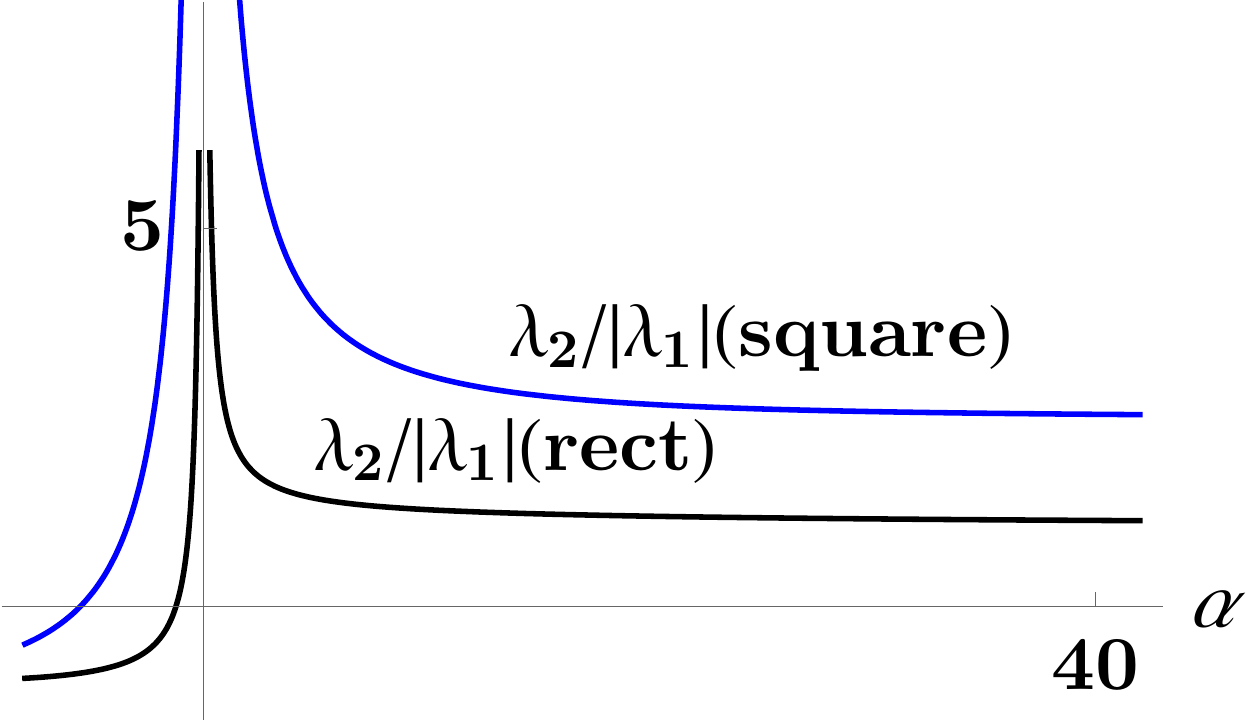}
\end{center}
\caption{\label{firsttwoperim} Length scaling $\alpha/L$. Left: the first two eigenvalues $\lambda_1(\cdot\,;\alpha/L)$ and $\lambda_2(\cdot\,;\alpha/L)$ for the unit square and for a rectangle with area $1$ and aspect ratio $7$, plotted as functions of $\alpha$. Here the perimeter is $L=4$ for the unit square and $L=2(\sqrt{7}+1/\sqrt{7})$ for the rectangle. \autoref{firsttwoperimcloseup} provides a close-up view near the origin. Right: the ratio $\lambda_2(\cdot\,;\alpha/L)/|\lambda_1(\cdot\,;\alpha/L)|$ for the square and rectangle.}
\end{figure}
\begin{figure}
\begin{center}
\includegraphics[scale=0.5]{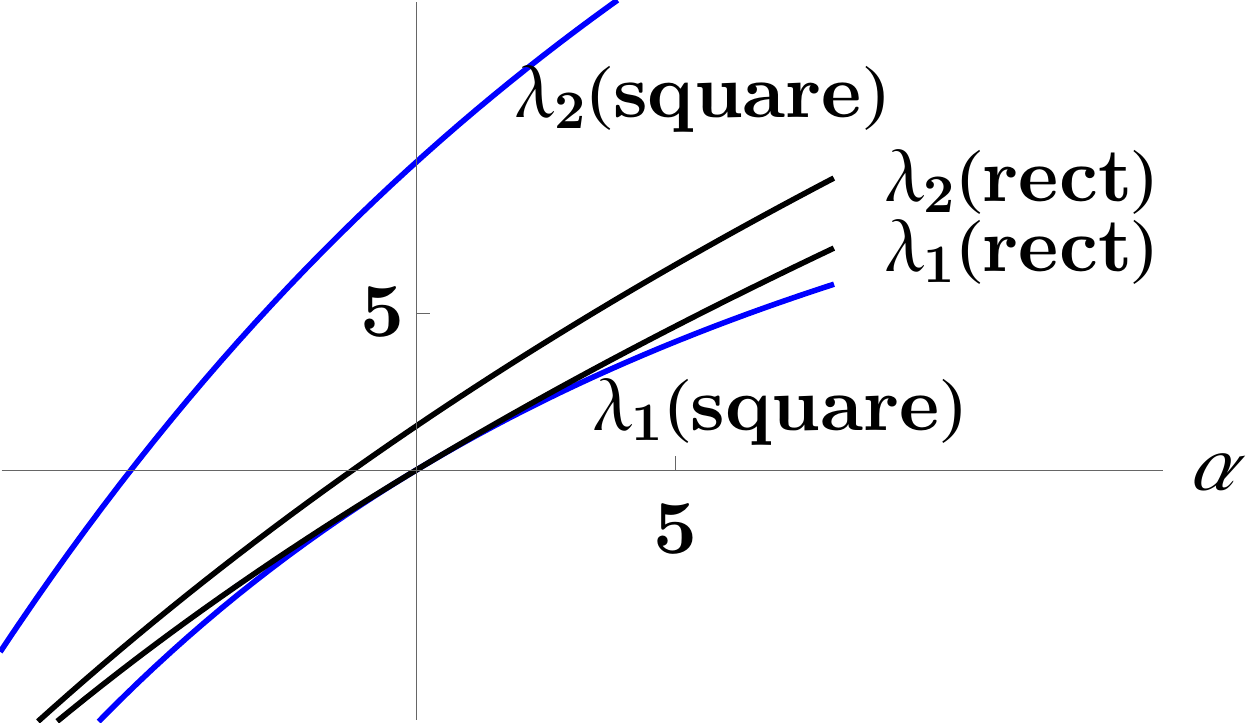}
\end{center}
\caption{\label{firsttwoperimcloseup} Length scaling $\alpha/L$. Close-up view near the origin of the left side of \autoref{firsttwoperim}, showing the first two eigenvalues $\lambda_1(\cdot\,;\alpha/L)$ and $\lambda_2(\cdot\,;\alpha/L)$ for the unit square and for a rectangle of area $1$. The first eigenvalue is minimal for the square, for all $\alpha$, and the second eigenvalue is maximal for the square in a range that includes the magnified region.}
\end{figure}
\begin{itemize}
\item minimality of the square for the first eigenvalue when $\alpha \in \R$ (\autoref{lambda1L}) 
\item maximality of the square for the second eigenvalue when $\alpha \in [\alpha_-,\alpha_+]$ (\autoref{lambda2L}), where $\alpha_- \simeq -9.4$ and $\alpha_+ \simeq 33.2$; outside that range the maximizer is the degenerate rectangle, 
\item maximality of the square for the first nonzero Steklov eigenvalue (\autoref{steklov}, which gives a new proof of a result by Girouard, Lagac\'{e}, Polterovich and Savo \cite{GLPS17})
\item maximality of the square for the spectral ratio $\lambda_2/\lambda_1$ (\autoref{ratio2dim}) when $\alpha>0$.
\end{itemize}
The first of these results, about minimality of the first Robin eigenvalue when the parameter is $\alpha/L$, suggests a new Rayleigh-type inequality, \autoref{lambda1Lgeneral}, in which the disk is the minimizer for the first eigenvalue among convex domains. This conjectured inequality applies for all $\alpha \in \R$, and notably does not switch direction at $\alpha=0$. The Bareket switching phenomenon seems not to occur, due to the scaling of the Robin parameter by perimeter. The second result, about maximizing the second eigenvalue, is the rectangular version of Freitas and Laugesen's \cite{FL18b} result for simply connected domains with $\alpha \in [-2\pi,2\pi]$. \autoref{convexconj} describes a higher dimensional generalization for the second eigenvalue on convex domains. The Steklov result (the third one above) is a rectangular Weinstock type inequality. The fourth result, maximality of the square for the spectral ratio, stimulates a conjecture for all convex domains when $\alpha \geq -2\pi$, in \autoref{ratioconjlength}. 

Lastly, the inverse spectral problem for Robin rectangles has an appealingly simple statement (\autoref{hearingdrum}): each rectangle is determined up to congruence by its first two Robin eigenvalues, whenever $\alpha \neq 0$. 

For background on spectral optimization for the Laplacian I recommend the survey by Grebenkov and Nguyen \cite{GN13}, and the book \cite{H17} edited by Henrot, which includes a chapter of Robin results. Upper and lower bounds on the first eigenvalue in terms of inradius have been developed by Kova\v{r}\'{\i}k \cite[Theorem 4.5]{Ko14}. His lower bound was recently sharpened by Savo \cite[Corollary 3]{S19}. For rectangular domains, the latest developments include an analysis of Courant-sharp Robin eigenvalues on the square by Gittins and Helffer \cite{GH19}, and of P\'{o}lya-type inequalities for disjoint unions of rectangles by Freitas and Kennedy \cite{FK18}. 

On a wry historical note, Robin's connection to the Robin boundary condition appears rather tenuous, according to investigations by Gustafson and Abe \cite{GA98}. 

\smallskip
The main results and conjectures are in the next two sections. Proofs appear later in the paper, especially in \autoref{mainproofs}. 

\section{\bf Monotonicity and concavity as a function of the Robin parameter}
\label{results}

We start by investigating the first two eigenvalues, and their gap and ratio, on a fixed domain as functions of the Robin parameter $\alpha$. Write $\Omega$ for a bounded Lipschitz domain in $\Rn$. The eigenvalues of the Robin Laplacian, denoted $\lambda_k(\Omega;\alpha)$ for $k=1,2,\dots$, are increasing and continuous as functions of the boundary parameter $\alpha \in \R$, and for each fixed $\alpha$ satisfy
\[
\lambda_1(\Omega;\alpha) < \lambda_2(\Omega;\alpha) \leq \lambda_3(\Omega;\alpha) \leq \dots \to \infty . 
\]
These facts can be established using the Rayleigh quotient and its associated minimax variational characterization of the $k$th eigenvalue, which reads
\begin{equation} \label{minimax}
\lambda_k(\Omega;\alpha) = \min_{U_k} \max_{u \in U_k \setminus 0} \frac{\int_\Omega |\nabla u|^2 \, dx + \alpha \int_{\partial \Omega} u^2 \, dS}{\int_\Omega u^2 \, dx} 
\end{equation}
where $U_k$ ranges over all $k$-dimensional subspaces of $H^1(\Omega)$; see for example \cite[{\S}4.2]{H17}.

\section*{Monotonicity} Each individual eigenvalue $\lambda_k(\Omega;\alpha)$ is increasing as a function of $\alpha$, by the minimax characterization \eqref{minimax}. Is the gap between the first two eigenvalues also increasing with respect to $\alpha$? On this question, Smits \cite{Sm96} has raised:
\begin{conjx}[Monotonicity of the spectral gap with respect to the Robin parameter;  \protect{\cite[Section 4]{Sm96}}] \label{NeumannDirichletGapConj}
For convex bounded domains $\Omega$, the spectral gap $(\lambda_2-\lambda_1)(\Omega;\alpha)$ is strictly increasing as a function of $\alpha > 0$. In particular, the Neumann gap provides a lower bound on the Dirichlet gap:
\[
\nu_2(\Omega) < \delta_2(\Omega) - \delta_1(\Omega) 
\]
where $0 = \nu_1 < \nu_2$ are the first and second Neumann eigenvalues of the Laplacian, and $0<\delta_1 < \delta_2$ are the first and second Dirichlet eigenvalues. 
\end{conjx}
Smits proved his conjecture for intervals, and observed that he had verified it also in $2$ dimensions for rectangles and disks. In the next theorem we provide a proof in all dimensions for rectangular boxes, handling both positive and negative values of $\alpha$. 

\autoref{NeumannDirichletGapConj} fails on some convex domains when $\alpha \ll 0$, since an asymptotic formula due to Khalile \cite[Corollary 1.3]{Kh18} implies that certain convex polygons with distinct smallest angles have spectral gap behaving like $(\text{const.})\alpha^2$ for large negative $\alpha$. The gap for such a polygon would be decreasing as a function of $\alpha$ when $\alpha \ll 0$. 

Rectangular boxes are better behaved, it turns out, for all real $\alpha$. 
\begin{theorem}[Monotonicity of the spectral gap with respect to the Robin parameter, on rectangular boxes] \label{gapmonot}
For a rectangular box $\mathcal{B}$, as $\alpha$ increases from $-\infty$ to $\infty$ the spectral gap $(\lambda_2-\lambda_1)(\mathcal{B};\alpha)$ strictly increases from $0$ to the Dirichlet gap $(\delta_2-\delta_1)(\mathcal{B})$. 
\end{theorem}

The spectral ratio $\lambda_2/\lambda_1$ satisfies a monotonicity property for all Lipschitz domains,  provided we multiply the ratio by $\alpha$. Write $\sigma_1(\Omega)$ for the first nonzero Steklov eigenvalue of the domain. 
\begin{theorem}[Monotonicity of $\alpha$ times the spectral ratio] \label{ratiomonot}
For each bounded Lipschitz domain $\Omega$, the map 
\[
\alpha \mapsto \alpha \frac{\lambda_2(\Omega;\alpha)}{\lambda_1(\Omega;\alpha)}
\]
is increasing for $\alpha > -\sigma_1(\Omega)$, that is, whenever $\lambda_2(\Omega;\alpha)>0$. 
\end{theorem}
The proof of \autoref{ratiomonot} breaks down when $\alpha < -\sigma_1(\Omega)$ (that is, when $\lambda_2$ is negative), although it seems reasonable to conjecture that $\alpha \lambda_2/\lambda_1$ is increasing for that range of $\alpha$-values also. 

The spectral ratio without a factor of $\alpha$ seems to decrease rather than increase. 
\begin{conjx}[Monotonicity of the spectral ratio] \label{gapmonotpure}
For every bounded Lipschitz domain $\Omega$, the map
\[
\alpha \mapsto \frac{\lambda_2(\Omega;\alpha)}{\lambda_1(\Omega;\alpha)}
\]
is decreasing for $\alpha>0$.  
\end{conjx}
This spectral ratio approaches $\infty$ as $\alpha \to 0$, since the first eigenvalue approaches zero. 

\autoref{gapmonotpure} is open even for rectangles, where the formula for the spectral ratio seems tricky to handle analytically. The conjecture can apparently fail for acute isosceles triangles with $\alpha<0$, by numerical work of D. Kielty (private communication). 

\section*{Concavity} Next, recall that the first Robin eigenvalue $\lambda_1(\Omega;\alpha)$ is a concave function of $\alpha \in \R$, by the Rayleigh principle 
\[
\lambda_1(\Omega;\alpha) = \min_{u \in H^1(\Omega)} \frac{\int_\Omega |\nabla u|^2 \, dx + \alpha \int_{\partial \Omega} u^2 \, dS}{\int_\Omega u^2 \, dx} ,
\]
which expresses the first eigenvalue curve as the minimum of a family of linear functions of $\alpha$. 

Is the second Robin eigenvalue also concave with respect to $\alpha$? That seems too much to expect on arbitrary domains, and even among convex domains, nonconcavity can occur for some negative $\alpha$ by numerical work of Kielty (private communication). So we raise a question for positive $\alpha$. 
\begin{conjx}[Concavity of the second eigenvalue with respect to the Robin parameter] \label{lambda2concavity}
For convex bounded domains $\Omega$, the second eigenvalue $\lambda_2(\Omega;\alpha)$ is a concave function of $\alpha>0$. 
\end{conjx}
A special case in which concavity holds is when the domain has a plane of symmetry and the second eigenfunction is odd with respect to that plane --- for then the second eigenvalue equals the first eigenvalue of the mixed Dirichlet--Robin problem on one half of the domain, and by the Rayleigh principle that mixed eigenvalue is concave with respect to $\alpha$. This argument applies, for example, to the ball and to rectangular boxes. (For a comprehensive treatment of the ball's spectrum, see \cite[Section 5]{FL18a}.) 

For rectangular boxes we may proceed explicitly, and handle all $\alpha \in \R$. 
\begin{theorem}[Concavity of the first two eigenvalues with respect to the Robin parameter, on rectangular boxes] \label{secondeigconcave}
For each rectangular box $\mathcal{B}$, the first and second eigenvalues $\lambda_1(\mathcal{B};\alpha)$ and $\lambda_2(\mathcal{B};\alpha)$ are strictly concave functions of $\alpha \in \R$. 
\end{theorem}

The gap $(\lambda_2-\lambda_1)(\mathcal{B};\alpha)$ appears to be concave too, for $\alpha>0$, according to  numerical investigations (omitted) that build on eigenvalue formulas for the box (as developed in \autoref{identifyinginterval} and \autoref{identifying}). I have not succeeded in proving this numerical observation. The gap cannot be concave for all $\alpha$, since the gap for the box is positive and tends to $0$ as $\alpha \to -\infty$, by \autoref{gapmonot}.

\section{\bf Optimal rectangular boxes for low Robin eigenvalues}
\label{optimal}

\section*{First Robin eigenvalue}

Faber and Krahn proved almost a century ago Rayleigh's conjecture that the first Dirichlet eigenvalue of the Laplacian is minimal on the ball, among domains of given volume. Bossel \cite{B86} proved the analogous result for Robin eigenvalues in $2$ dimensions with $\alpha>0$, by applying her new characterization of the first eigenvalue. Daners \cite{D06} extended Bossel's method to higher dimensions. An alternative approach via the calculus of variations was found more recently by Bucur and Giacomini \cite{BG10,BG15a}. 

For $\alpha < 0$, Bareket \cite{B77} conjectured that the ball would maximize (not minimize) the Robin eigenvalue among domains of given volume. Although Bareket's conjecture turns out to be false for large $\alpha<0$ due to an annular counterexample by Freitas and \Krejcirik\ \cite{FK15}, those authors did establish the conjecture in $2$ dimensions whenever $\alpha<0$ is small enough, depending only on the volume of the domain. Bareket's conjecture holds also when the domain is close enough to a ball, by Ferone, Nitsch and Trombetti \cite{FNT15}, and holds for all domains in $2$-dimensions if perimeter rather than area of the domain is normalized, by Antunes, Freitas and \Krejcirik\ \cite[Theorem 2]{AFK17}.  

For the class of rectangular boxes, the analogue of the Bossel--Daners theorem was proved by Freitas and Kennedy \cite[Theorem 4.1]{FK18} in $2$ dimensions and in all dimensions by Keady and Wiwatanapataphee \cite{KW18} (whose paper contains several other results for rectangles too). That is, they proved $\lambda_1(\mathcal{B};\alpha)$ is minimal for the cube among all rectangular boxes $\mathcal{B}$ of given volume, when $\alpha>0$. We state this result next as part (i), and prove also in part (ii) a reversed or Bareket-type inequality for all $\alpha<0$. 

\begin{theorem}[Extremizing the first Robin eigenvalue on rectangular boxes]\label{lambda1higherdim}\ 

\noindent (i) If $\alpha>0$ then $\lambda_1(\mathcal{B};\alpha)$ is positive and is minimal for the cube and only the cube, among rectangular boxes $\mathcal{B}$ of given volume. 

\noindent (ii) If $\alpha<0$ then $\lambda_1(\mathcal{B};\alpha)$ is negative and is maximal for the cube and only the cube, among rectangular boxes $\mathcal{B}$ of given volume. 
\end{theorem}
The ``opposite'' extremal problems are known to have no solution, since when $\alpha>0$ the first eigenvalue $\lambda_1(\mathcal{B};\alpha)$ is unbounded above as the rectangular box degenerates in some  direction, and when $\alpha<0$ the eigenvalue is unbounded below; see \eqref{firsttoinfinity} later in the paper, for these facts. 

Next we restrict attention to rectangles $\mathcal{R}$ in $2$ dimensions. Scaling the Robin parameter by boundary length is found to yield a different (and, in some cases, better) result of Rayleigh type for the first eigenvalue. To state the result, we write 
\[
\text{$A=$ area of rectangle $\mathcal{R}$, \quad $L=$ length of $\partial \mathcal{R}$.}
\]
The quantity 
\[
\lambda_1(\mathcal{R};\alpha/L)A
\]
is scale invariant by an easy rescaling calculation, and it is known to be minimal among rectangles for the square when $\alpha=\infty$ (the Dirichlet case), when $\alpha=0$ (the trivial Neumann case), and when $\alpha \to -\infty$ (by substituting into Levitin and Parnovski's asymptotic for piecewise smooth domains \cite[Theorem 4.15]{H17},  \cite[{\S}3]{LP08}). Thus it is natural to suspect the square should be the minimizer for each real value of $\alpha$, including for $\alpha<0$. 
\begin{theorem}[Minimizing the first Robin eigenvalue on rectangles, with length scaling]\label{lambda1L} If $\alpha \neq 0$ then $\lambda_1(\mathcal{R};\alpha/L)A$ is minimal for the square and only the square, among rectangles $\mathcal{R}$. 
\end{theorem}
When $\alpha=0$, every domain has the same first Neumann eigenvalue, namely $\lambda_1=0$.

\autoref{lambda1L} implies the $2$-dimensional case of \autoref{lambda1higherdim} when $\alpha>0$, as follows. Suppose $\mathcal{R}$ and $\mathcal{S}$ are a rectangle and a square having the same area $A$. Since $\lambda_1$ is increasing with respect to the Robin parameter, the rectangular isoperimetric inequality $4\sqrt{A} \leq L$ implies that 
\[
\lambda_1(\mathcal{R};\alpha/4\sqrt{A}) \geq \lambda_1(\mathcal{R};\alpha/L) \geq \lambda_1(\mathcal{S};\alpha/L(\mathcal{S})) ,
\]
with the final inequality holding by \autoref{lambda1L}. Since $L(\mathcal{S})=4\sqrt{A}$, we conclude $\lambda_1(\mathcal{R};\alpha/4\sqrt{A}) \geq \lambda_1(\mathcal{S};\alpha/4\sqrt{A})$. Replacing $\alpha$ by $4\sqrt{A} \alpha$, we deduce $\lambda_1(\mathcal{R};\alpha) \geq \lambda_1(\mathcal{S};\alpha)$ for all $\alpha>0$, which is  \autoref{lambda1higherdim}(i) in $2$ dimensions. Incidentally, this argument reveals that the scale invariant form of \autoref{lambda1higherdim}(i), if one wants to write it that way, involves minimizing $\lambda_1(\mathcal{R};\alpha/\sqrt{A})A$ in $2$ dimensions and in higher dimensions minimizing $\lambda_1(\mathcal{B};\alpha/V^{1/n})V^{2/n}$, where $V$ is the volume of the box $\mathcal{B}$.

One would like to extend \autoref{lambda1L} to higher dimensions for $\lambda_1(\mathcal{B};\alpha/S^{1/(n-1)})V^{2/n}$, or even better for $\lambda_1(\mathcal{B};\alpha V^{1-2/n}/S)V^{2/n}$, where $S$ the surface area of the box. I have not succeeded in establishing these generalizations.  

For arbitrary convex domains, an improved Rayleigh--Bossel type conjecture is suggested by \autoref{lambda1L}.  
\begin{conjx}[Minimizing the first Robin eigenvalue on convex domains, with length scaling] \label{lambda1Lgeneral}
For each $\alpha \in \R$, the scale invariant ratio
\[
    \lambda_1(\Omega;\alpha/L(\Omega)) A(\Omega)
\]
is minimal when the convex bounded planar domain $\Omega$ is a disk. 
\end{conjx}
The conjecture arose in conversation with P. Freitas, and is stated in our work \cite{FL18b}.

For $\alpha<0$, this conjecture goes in the opposite direction to the upper bound conjectured by Bareket, which does not employ length scaling on the Robin parameter. 

For $\alpha>0$, the conjecture would strengthen Bossel's theorem in the class of convex domains, as one sees by replacing $L(\Omega)$ with $\sqrt{A(\Omega)}$ and arguing with rescaling like we did for rectangles above. 

\autoref{lambda1Lgeneral} is known to hold for $\alpha=\infty$ (the usual Faber--Krahn inequality), and trivially for $\alpha=0$, and it holds also on smooth domains as $\alpha \to -\infty$  since 
\begin{equation} \label{Robinasymptotic}
\lambda_1(\Omega;\alpha/L(\Omega)) A(\Omega) \sim - \frac{A(\Omega)}{L(\Omega)^2} \alpha^2 \qquad \text{as $\alpha \to -\infty$}
\end{equation}
by the Robin asymptotic of Lacey \emph{et al.}\ \cite[Theorem 4.14]{H17}, \cite{LOS98}, noting $A/L^2$ is maximal for the disk by the isoperimetric theorem. 

\autoref{lambda1Lgeneral} can fail for nonconvex domains, as Dorin Bucur pointed out to me on a sunny morning during a conference in Santiago, by using boundary perturbation arguments. For $\alpha>0$ one can drive $\lambda_1(\Omega;\alpha/L(\Omega))$ arbitrarily close to $0$ by imposing a boundary perturbation that greatly increases the perimeter and barely changes the area. For example, one could add to the domain an outward spike of width $\epsilon^2$ and length $1/\epsilon$ and then construct a trial function that equals $1$ on the original domain and vanishes on the spike, except for a transition zone of length $1$; the Rayleigh quotient is then $O(\epsilon)$ as $\epsilon \to 0$. For $\alpha<0$ one can drive $\lambda_1(\Omega;\alpha/L(\Omega))$ arbitrarily close to $-\infty$ by doing the same spike perturbation except taking the complementary trial function, that is, the function that vanishes on the original domain and equals $1$ on the spike except for the transition zone of length $1$; its Rayleigh quotient equals $\alpha/\epsilon+O(1)$ as $\epsilon \to 0$.

In the reverse direction to the conjecture, a sharp \emph{upper} bound on the first eigenvalue is known for general domains, when the Robin parameter is scaled by boundary length.
\begin{theorem}[Maximizing the first Robin eigenvalue, with length scaling; see \protect{\cite[Theorem A]{FL18b}}] \label{linearbound}
Fix $\alpha \neq 0$. If $\Omega$ is a bounded, Lipschitz planar domain then 
\[
\lambda_1\big( \Omega;\alpha/L(\Omega) \big) A(\Omega) < \alpha   
\]
with equality holding in the limit for rectangular domains that degenerate to a line segment (meaning the aspect ratio tends to infinity). More generally, if $\Omega \subset \Rn, n \geq 2$ is a bounded Lipschitz domain then 
\[
\lambda_1\big(\Omega;\alpha V(\Omega)^{1-2/n} /S(\Omega)\big) V(\Omega)^{2/n} < \alpha   
\]
with equality holding in the limit for degenerate rectangular boxes (which means as $S/V^{(n-1)/n} \to \infty$). 
\end{theorem}
In the omitted case $\alpha=0$, all domains have first Neumann eigenvalue $\lambda_1(\Omega;0) = 0$. 

Note that \autoref{linearbound} is sharp for fixed $\alpha$, but not sharp for a fixed domain $\Omega$ as $\alpha \to \pm \infty$, since the first Robin eigenvalue approaches a finite number (the Dirichlet eigenvalue) as $\alpha \to \infty$ and approaches $-\infty$ quadratically rather than linearly as $\alpha \to -\infty$, by the asymptotic formula \eqref{Robinasymptotic}. 

\section*{Second Robin eigenvalue}

The second Dirichlet eigenvalue is minimal for the union of two balls, under a volume constraint; this observation by Krahn \cite{K26} was extended by Kennedy \cite{K09} from the Dirichlet to the Robin case for $\alpha>0$. The survey article \cite[{\S}4.6.1]{H17} gives a clear account of these lower bounds, which are applications of the Faber--Krahn and Bossel--Daners theorems, respectively. 

Upper bounds do not exist for $\alpha>0$. The second Dirichlet eigenvalue has no upper bound, since a thin rectangular box of given volume can have arbitrarily large eigenvalue. The same reasoning holds in the Robin case when $\alpha>0$, as was remarked after \autoref{lambda1higherdim}. 

For $\alpha=0$, the second Neumann eigenvalue does have an upper bound, being largest for the ball by work of Szeg\H{o} \cite{S54} for simply connected domains in the plane and Weinberger \cite{W56} for domains in all dimensions. This Neumann result was extended recently to the second Robin eigenvalue for a range of $\alpha \leq 0$ by Freitas and Laugesen \cite{FL18a}. Specifically, they proved $\lambda_2(\Omega;\alpha)$ is maximal for the ball $B(R)$ having the same volume as $\Omega$, for each $\alpha \in \big[-(1+1/n)R^{-1},0 \big]$. The result fails when $\alpha<0$ is large in magnitude, by an annular counterexample. Corollaries include Weinberger's result for the Neumann eigenvalue ($\alpha=0$), and Brock's sharp upper bound \cite{B01} on the first nonzero Steklov eigenvalue, which follows from taking $\alpha=-R^{-1}$. 

The analogous assertions for rectangular boxes hold for all $\alpha \leq 0$:
\begin{theorem}[Maximizing the second Robin eigenvalue on rectangular boxes]\label{lambda2higherdim}
If $\alpha \leq 0$ then $\lambda_2(\mathcal{B};\alpha)$ is maximal for the cube and only the cube, among rectangular boxes $\mathcal{B}$ of given volume. 
\end{theorem}
\begin{corollary}[Maximizing the first nonzero Steklov eigenvalue on rectangular boxes]\label{sigma1higherdim}
The first nonzero Steklov eigenvalue $\sigma_1(\mathcal{B})$ is maximal for the cube and only the cube, among rectangular boxes $\mathcal{B}$ of given volume. 
\end{corollary}
The Steklov result in \autoref{sigma1higherdim} was proved directly by Girouard \emph{et al.}\ \cite[Theorem 1.6]{GLPS17}. Indeed, they proved a stronger result, namely that the cube maximizes $\sigma_1$ among rectangular boxes of given surface area.  

\subsubsection*{Length-scaled Robin parameter.}
An upper bound on the second Robin eigenvalue with length-scaled Robin parameter was proved by Freitas and Laugesen \cite[Theorem B]{FL18b} for simply connected planar domains, namely that $\lambda_2(\Omega;\alpha/L) A$ is maximal for the disk provided $\alpha \in [-2\pi,2\pi]$. (It is not known to what extent that interval of $\alpha$-values can be enlarged.) Thanks to the isoperimetric inequality this result implies, for simply connected domains with $\alpha \in [-R^{-1},0]$, the inequality from \cite{FL18a} that $\lambda_2(\Omega;\alpha)$ is maximal for the disk of the same area. It also implies Weinstock's result \cite{We54} that the first nonzero Steklov eigenvalue of a simply connected domain is maximal for the disk, under perimeter normalization, as explained in \cite{FL18b}. 

It is an open problem to generalize this length-scaled upper bound on the second eigenvalue to higher dimensions. Convexity might provide a reasonable substitute for simply connectedness. Write  $\B$ for the unit ball. The next conjecture was raised by Freitas and Laugesen \cite{FL18b}. 
\begin{conjx}[Maximizing the second Robin eigenvalue on convex domains, with surface area and volume scaling \protect{\cite[Conjecture 2]{FL18b}}]\label{convexconj}
The ball maximizes the scale invariant quantity $\lambda_2(\Omega;\alpha V^{1-2/n}/S) V^{2/n}$
among all convex bounded domains $\Omega \subset \Rn$, for $\alpha$ such that $-1 \leq \alpha V(\B)^{1-2/n}/S(\B) \leq 0$. Hence the ball also maximizes $\lambda_2(\Omega;\alpha/S^{1/(n-1)}) V^{2/n}$ for a suitable range of $\alpha$, where now the Robin parameter is scaled purely by perimeter. 
\end{conjx}
Taking $n=2$ reduces the conjecture back to maximizing $\lambda_2(\Omega;\alpha/L)A$. 

For rectangles, we will develop a length-scaled upper bound on the second eigenvalue that is analogous to \autoref{convexconj} and has a sharp interval of $\alpha$-values. Let $\alpha_+ \simeq 33.2054$ and $\alpha_- \simeq -9.3885$ be the numbers defined later by \eqref{alphaplus} and \eqref{alphaneg}.
\begin{theorem}[Maximizing the second Robin eigenvalue on rectangles, with length scaling]\label{lambda2L}
If $\alpha \in [\alpha_-,\alpha_+]$ then $\lambda_2(\mathcal{R};\alpha/L)A$ is maximal for the square and only the square, among rectangles $\mathcal{R}$. 

If $\alpha \notin [\alpha_-,\alpha_+]$ then the degenerate rectangle is asymptotically maximal among all rectangles,  meaning $\lambda_2(\mathcal{R};\alpha/L)A < \alpha$ with equality in the limit as $\mathcal{R}$ degenerates to an interval.
\end{theorem}
One would like to generalize to boxes in higher dimensions, but I have not succeeded in doing so. 

\autoref{lambda2L} implies the $2$-dimensional case of \autoref{lambda2higherdim} with the Robin parameter replaced by $\alpha/4\sqrt{A}$, when $\alpha \in [\alpha_-,0]$: one simply argues like we did earlier for the first eigenvalue after the statement of \autoref{lambda1L}, using that $\lambda_2$ is increasing with respect to the Robin parameter and that $\alpha/4\sqrt{A} \leq \alpha/L$ by the rectangular isoperimetric inequality, when $\alpha \leq 0$.

A corollary of \autoref{convexconj} would be a result proved already by Bucur \emph{et al.}\ \cite[Theorem 3.1]{BFNT17} that among convex domains, the ball maximizes the scale invariant quantities $\sigma_1 S/V^{1-2/n}$ and $\sigma_1 S^{1/(n-1)}$. Here $\sigma_1$ is the first nonzero Steklov eigenvalue; recall the Steklov eigenfunctions are harmonic and satisfy $\partial u/\partial \nu = \sigma u$ on the boundary, with eigenvalues $0=\sigma_0 < \sigma_1 \leq \sigma_2 \leq \dots$. 

Analogously, \autoref{lambda2L} has as a corollary that the first nonzero Steklov eigenvalue of a rectangle is maximal for the square, under perimeter normalization. 
\begin{corollary}[Maximizing the first nonzero Steklov eigenvalue on rectangles, with length normalization] \label{steklov}
The scale invariant quantity $\sigma_1 L$ is maximal among rectangles for the square and only the square. 
\end{corollary}
\autoref{steklov} is not new. It is due to Girouard \emph{et al.}\ \cite[Theorem 1.6]{GLPS17}, who proved the result and its extension to all dimensions, showing that that the cube maximizes $\sigma_1$ among rectangular boxes of given surface area. See also Tan \cite{T17} for the $2$-dimensional case of rectangles. What is new is our derivation of the Steklov corollary from a family of Robin results.  

\begin{remark}
For simply connected domains we recalled above that $\lambda_2(\Omega;\alpha/L)A$ is maximal for the disk when $-2\pi \leq \alpha \leq 2\pi$. Perhaps at some $\alpha$-value beyond $2\pi$ another domain  takes over as maximizer, and so on again and again as $\alpha$ continues to increase toward infinity? In the class of rectangles, at least, such ``domain cascading'' does not occur. Instead, \autoref{lambda2L} establishes a sharp transition between the square and the degenerate rectangle precisely at the Robin parameters $\alpha_-$ and $\alpha_+$. 
\end{remark}

\section*{Spectral gap $\lambda_2-\lambda_1$}

The Neumann and Dirichlet spectral gaps are minimal for the line segment among all convex domains in $\Rn$ of given diameter, by work of Payne--Weinberger \cite{PW60} and Andrews--Clutterbuck \cite{AC11}, respectively. For the Robin gap, an analogous conjecture has been stated by Andrews, Clutterbuck and Hauer \cite{ACH18}: 
\begin{conjx}[Minimizing the spectral gap on convex domains, under diameter normalization \protect{\cite[Sections 2 and 10]{ACH18}}]\label{robingap} Fix $\alpha > 0$ and the dimension $n \geq 2$. Among convex bounded domains $\Omega \subset \Rn$ of given diameter $D$, the Robin spectral gap is minimal for the degenerate box (line segment) of diameter $D$: 
\[
\lambda_2(\Omega;\alpha)-\lambda_1(\Omega;\alpha) > \lambda_2((0,D);\alpha)-\lambda_1((0,D);\alpha) .
\]
\end{conjx}
A partial result \cite[Theorem 2.1]{ACH18} says that the inequality holds with $\alpha$ on the right side replaced by $0$, that is, replacing the right side by the Neumann gap $\pi^2/D^2$; and even this result assumes the Robin ground state on $\Omega$ is log-concave, which is known to fail for some convex domains \cite[Theorem 1.2]{ACH18}. 

For rectangular boxes, we can prove the Robin gap conjecture for all $\alpha \in \R$.
\begin{theorem}[Minimizing the spectral gap on rectangular boxes, under diameter normalization]\label{gapD} Fix $\alpha \in \R$ and the dimension $n \geq 2$. Among rectangular boxes $\mathcal{B}$ of given diameter $D$, the Robin spectral gap is minimal for the degenerate box (line segment) of diameter $D$: 
\[
\lambda_2(\mathcal{B};\alpha)-\lambda_1(\mathcal{B};\alpha) > \lambda_2((0,D);\alpha)-\lambda_1((0,D);\alpha) .
\]
\end{theorem}

Maximizing the Robin gap is generally not possible among convex domains of given diameter, perimeter, or area, since the Dirichlet spectral gap can be arbitrarily large by an observation of Smits \cite[Theorem 5 and discussion]{Sm96}. He worked with a degenerating family of sectors. A degenerating family of acute isosceles triangles would presumably behave the same way. 

Among rectangular boxes, though, the spectral gap is not only bounded above, it is maximal at the cube for each value of the Robin parameter. 
\begin{theorem}[Maximizing the spectral gap on rectangular boxes]\label{gapSV} Fix $\alpha \in \R$ and the dimension $n \geq 2$. Among rectangular boxes $\mathcal{B}$ of given diameter (or given surface area, or given volume), the Robin spectral gap $(\lambda_2-\lambda_1)(\mathcal{B};\alpha)$ is maximal for the cube and only the cube.  
\end{theorem}

\section*{Spectral ratio $\lambda_2/\lambda_1$} 

The gap maximization in \autoref{gapSV} allows us to maximize also the ratio of the first two eigenvalues. We take the absolute value of the first eigenvalue, in the next result, in order to unify the cases of positive and negative $\alpha$. 

\begin{corollary}[Maximizing the spectral ratio on rectangular boxes]\label{ratio} Fix $\alpha \neq 0$ and the dimension $n \geq 2$. Among rectangular boxes $\mathcal{B}$ of given volume, the  Robin spectral ratio 
\[
\frac{\lambda_2(\mathcal{B};\alpha)}{|\lambda_1(\mathcal{B};\alpha)|}
\]
is maximal for the cube and only the cube.  
\end{corollary}
\begin{corollary}[Maximizing the spectral ratio on rectangles, with length scaling]\label{ratio2dim} Fix $\alpha > 0$. The length-scaled Robin spectral ratio 
\[
\frac{\lambda_2(\mathcal{R};\alpha/L)}{|\lambda_1(\mathcal{R};\alpha/L)|}
\]
is maximal for the square and only the square, among rectangles $\mathcal{R}$.  
\end{corollary}
The absolute value on $\lambda_1$ is superfluous in the statement of \autoref{ratio2dim}, since the first eigenvalue is positive when $\alpha>0$. We retain the absolute value anyway because the corollary ought to hold also when $\alpha<0$ --- although I have not found a proof. 

If the \autoref{gapmonotpure} for monotonicity of the spectral ratio were known to be true, then \autoref{ratio2dim} would imply the planar case of \autoref{ratio}, for $\alpha>0$. That short argument is left to the reader. 

The spectral ratio has a long history. Payne, P\'{o}lya and Weinberger \cite{PPW55} proved in the Dirichlet case ($\alpha=\infty$) that $\lambda_2/\lambda_1 \leq 3$ for planar domains. Payne and Schaefer \cite[\S3]{PS01} extended that result to hold on an interval of $\alpha$-values near $\infty$. The Payne--P\'{o}lya--Weinberger (PPW) conjecture asserted a sharp upper bound: that the Dirichlet ratio $\lambda_2/\lambda_1$ should be maximal for the disk. This conjecture and its analogue in $n$ dimensions were proved by Ashbaugh and Benguria \cite{AB91}. The analogous Robin question has been raised by Henrot \cite[p.~458]{H03}: to find the range of $\alpha$ values for which the ball maximizes the Robin spectral ratio. Some inequalities on that ratio have been proved by Dai and Shi \cite{DS14}.

In view of these ratio results, an analogue of \autoref{ratio} seems plausible for general domains. 
\begin{conjx}[Maximizing the spectral ratio]\label{ratioconj} Fix the dimension $n \geq 2$. Among bounded Lipschitz domains $\Omega$ of given volume, the Robin spectral ratio 
\[
\frac{\lambda_2(\Omega;\alpha)}{|\lambda_1(\Omega;\alpha)|}
\]
is maximal for the ball $B(R)$ having the same volume as $\Omega$, when $\alpha \geq -1/R, \alpha \neq 0$.
\end{conjx}
\autoref{ratioconj} holds for sufficiently small $\alpha<0$ on $C^2$-smooth planar domains of given area, because in that situation Freitas and \Krejcirik\ \cite[Theorem 2]{FK15} showed $|\lambda_1(\Omega;\alpha)|$ is minimal for the disk (the Bareket conjecture), while $\lambda_2(\Omega;\alpha)$ is maximal for the disk by a result of Freitas and Laugesen \cite[Theorem A]{FL18a}. 

The conjecture holds also at $\alpha=-1/R$ in all dimensions, since in that case the spectral ratio is $\leq 0$ by \cite[Theorem A]{FL18a}, with equality for the ball. 

The limiting case $\alpha \to 0$ of \autoref{ratioconj} follows from the isoperimetric theorem and the Szeg\H{o}--Weinberger theorem \cite{W56} for the first nonzero Neumann eigenvalue, as we now explain. For $\alpha \simeq 0$ one has $\lambda_1(\Omega;\alpha) \simeq \alpha S/V$ where $S$ is the surface area of $\partial \Omega$ and $V$ is the volume of $\Omega$ (see \cite[p.\,89]{H17}). Also $\lambda_2(\Omega;\alpha) \simeq \lambda_2(\Omega;0)=\nu_1(\Omega)$, the first nonzero Neumann eigenvalue. Thus \autoref{ratioconj} says in the limit $\alpha \to 0$ that $\nu_1(\Omega)/S$ is maximal for the ball. The isoperimetric theorem guarantees $S$ is minimal for the ball, and the Szeg\H{o}--Weinberger theorem gives maximality of $\nu_1$ for the ball, among domains of given volume. Hence this limiting case of the conjecture holds true. 

For $\alpha < -1/R$, I am not sure what domain might extremize the spectral ratio. Any extremal conjecture would need to be consistent with the spectral asymptotics as $\alpha \to -\infty$. For the ball or any other smooth domain, $\lambda_1$ and $\lambda_2$ are known to behave like $-\alpha^2$ to leading order (by Lacey \emph{et al.}\ \cite{LOS98} for the first eigenvalue and Daners and Kennedy \cite{DK10} for all eigenvalues; see \cite[{\S}4]{H17} for more literature). Thus $\lambda_2/|\lambda_1| \to -1$ as $\alpha \to -\infty$. On the other hand, the asymptotics for polygonal domains by Khalile \cite[Corollary 1.3, Theorem 3.6]{Kh18} imply that certain convex polygons have spectral ratio $\lambda_2/|\lambda_1|$ converging to a constant greater than $-1$ as $\alpha \to -\infty$. 

One lesson here is that rectangles provide an unreliable guide to the behavior of general domains, for large negative $\alpha$. One should in that range consider at least polygons whose angles are not all the same.  

We finish this subsection by conjecturing an analogue of \autoref{ratio2dim}. 
\begin{conjx}[Maximizing the spectral ratio on convex domains, with length scaling]\label{ratioconjlength} Among convex bounded planar domains $\Omega$, the length-scaled Robin spectral ratio 
\[
\frac{\lambda_2(\Omega;\alpha/L)}{|\lambda_1(\Omega;\alpha/L)|}
\]
is maximal for the disk, for each $\alpha \geq -2\pi$. 
\end{conjx}
The conjecture holds when $\alpha=-2\pi$, because then the second eigenvalue is $\leq 0$ with equality for the disk, by \cite[Theorem B]{FL18b} (which applies to all simply connected planar domains, not just convex ones). Further, the second eigenvalue of the disk is positive when $\alpha>-2\pi$. 

The limiting case $\alpha \to 0$ of \autoref{ratioconjlength} reduces to the Szeg\H{o}--Weinberger theorem, since $\lambda_1(\Omega;\alpha/L) \simeq \alpha/A$ and $\lambda_2(\Omega;\alpha/L) \simeq \nu_1(\Omega)$. 

The limiting case $\alpha \to \infty$ of the conjecture would recover the convex planar case of  Ashbaugh and Benguria's sharp PPW inequality.

\section*{Hearing the shape of a Robin rectangle} Dirichlet and Neumann drums cannot always be ``heard'', as Gordon, Webb and Wolpert \cite{GWW92} famously showed. The inverse spectral problem for Robin drums is apparently an open problem. Arendt, ter Elst and Kennedy \cite{AEK14} have written that ``it may well be the case that one can hear the shape of a drum after all, if one loosens the membrane before striking it''.

Hearing the shape of a \emph{rectangular} drum with Robin boundary conditions is a solvable special case, and requires merely the first two frequencies:
\begin{theorem}[Hearing a rectangular Robin drum]\label{hearingdrum}
If $\alpha \neq 0$ then each rectangle $\mathcal{R}$ is determined up to congruence by its first two eigenvalues, $\lambda_1(\mathcal{R};\alpha)$ and $\lambda_2(\mathcal{R};\alpha)$.
\end{theorem}
The theorem is spectacularly false in the Neumann case ($\alpha = 0$), where no pre-specified number of eigenvalues can be guaranteed to determine the rectangle. For example, every rectangle of width $m$ and height less than $1$ has the same first $m+1$ Neumann eigenvalues, namely $(j\pi/m)^2$ for $j=0,1,\dots,m$. 

Incidentally, a Steklov inverse spectral problem was resolved recently for rectangular boxes in all dimensions by Girouard \emph{et al.}\ \cite[Corollary 1.8]{GLPS17}, who observed that the full spectrum determines the perimeter, and then the perimeter and the first eigenvalue $\sigma_1$ together determine the rectangle. 

\section*{Polygonal open problems}
\label{openproblems}

For each theorem where the square is the optimizer among rectangles, it seems reasonable to conjecture that the square is in fact optimal among all (convex) quadrilaterals. The exception is \autoref{gapSV}, where the spectral gap is unbounded above in general; see the discussion before that theorem. 

The equilateral triangle should presumably be optimal among triangles, although sometimes triangles are so  ``pointy'' that they behave differently from general domains. More generally, the regular $N$-gon might be optimal among (convex) $N$-gons, although such problems seem currently out of reach --- for example, the polygonal Rayleigh conjecture about minimizing the first Dirichlet eigenvalue remains open even for pentagons. 

The inverse spectral problem for triangles is particularly fascinating. A triangle is known to be determined by its full Dirichlet spectrum, via the wave trace method of Durso \cite{D90}. Later, Grieser and Maronna \cite{GM13} found a delightful, different proof using the heat trace and the sum of reciprocal angles of the triangle. These results are wildly overdetermined, though, since they employ infinitely many eigenvalues in pursuit of the three side lengths of the triangle. For that reason, Laugesen and Suideja \cite[p.\,17]{LS09} suspected that the first three Dirichlet eigenvalues should suffice to determine a triangle. Antunes and Freitas \cite{AF11} developed convincing numerical evidence in favor of that conjecture, although a proof remains elusive. Similar results should presumably hold for the Robin problem when $\alpha \neq 0$. (The Neumann case $\alpha=0$ is less clear \cite[Section 3c]{AF11}.) No investigations appear yet to have been carried out on determining a triangle from its first three Robin eigenvalues. 

\section{\bf Monotonicity and convexity lemmas}
\label{notation}

This self-contained section establishes the underpinnings of the rest of the paper. The section can be skipped for now, and revisited later as needed. 

The four basic functions needed to determine the first and second Robin eigenvalues of intervals, and hence of rectangular boxes, are: 
\begin{align*}
g_1(x) = x \tan x , \qquad & g_1 : (0,\pi/2) \to (0,\infty) , \\
g_2(x) = -x \cot x , \qquad & g_2 : (0,\pi) \to (-1,\infty) , \\
h_1(x) = x \tanh x ,  \qquad & h_1 : (0,\infty) \to (0,\infty) , \\
h_2(x) = x \coth x ,  \qquad & h_2 : (0,\infty) \to (1,\infty) . 
\end{align*}
These functions have positive first derivatives and so are strictly increasing, as shown in \autoref{g1g2h1h2}. 
\begin{figure}
\begin{center}
\includegraphics[scale=0.4]{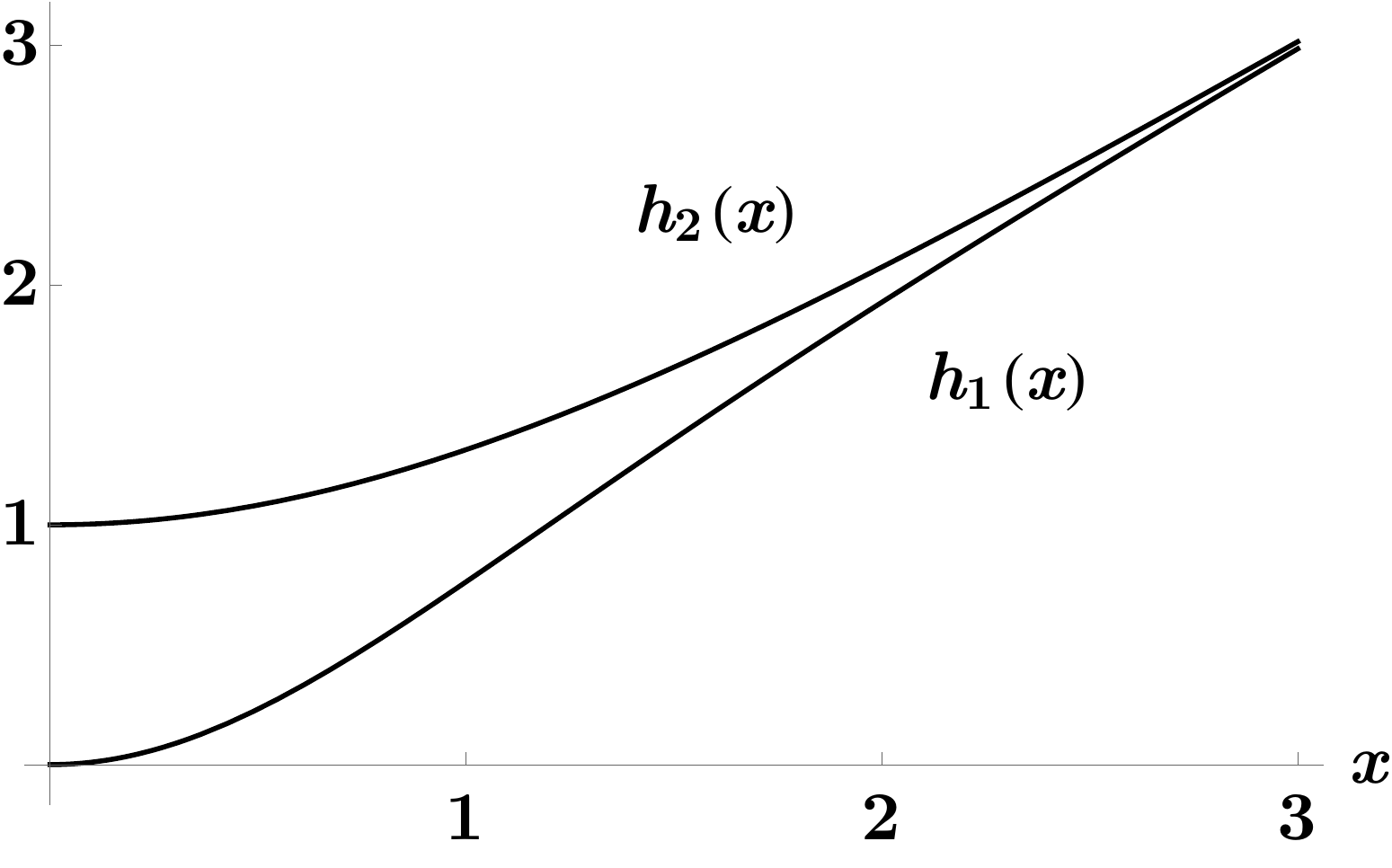}
\hspace{1cm}
\includegraphics[scale=0.4]{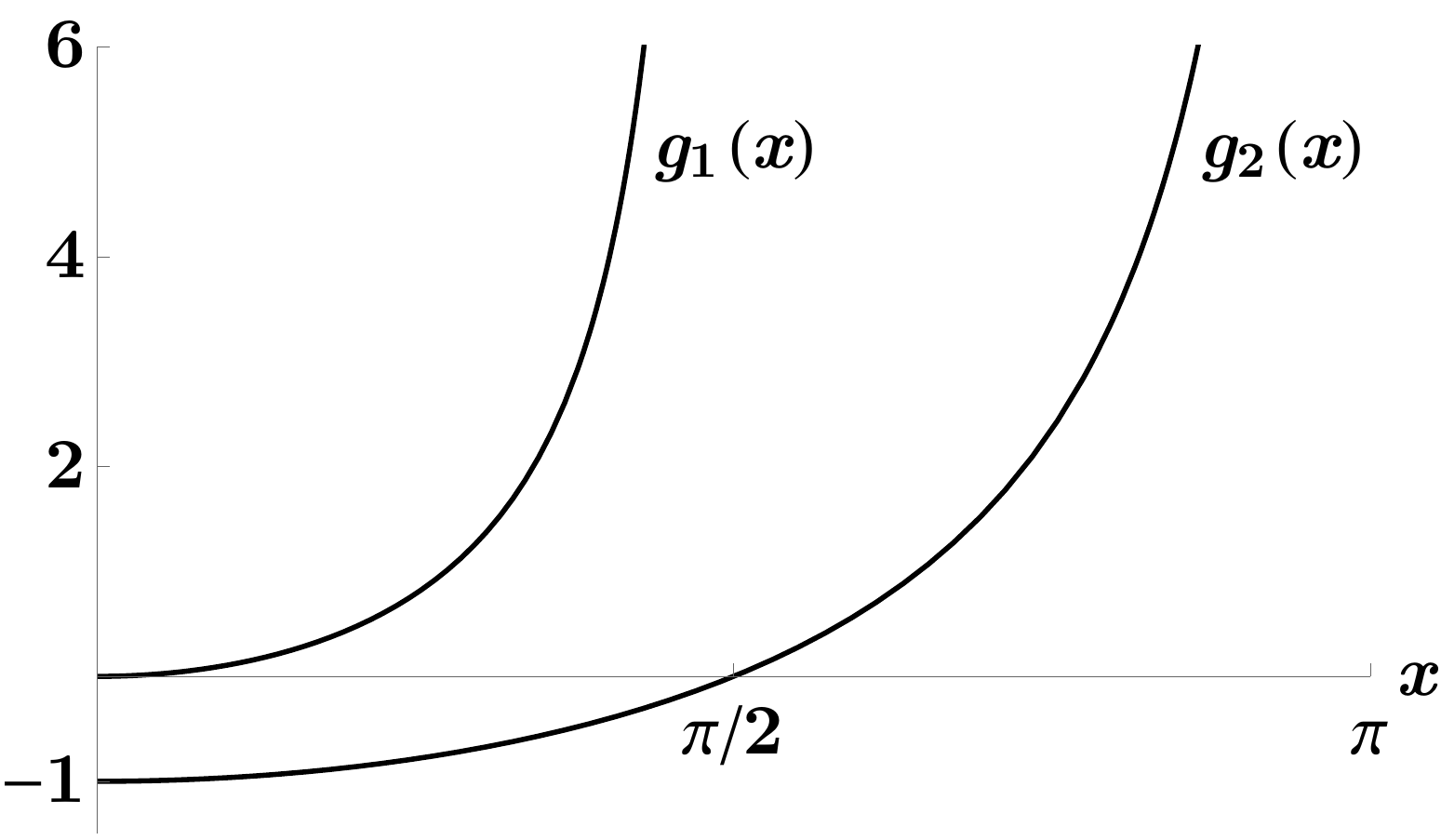}
\end{center}
\caption{\label{g1g2h1h2} The functions $h_1(x)=x \tanh x$ and $h_2(x)=x \coth x$, and $g_1(x)=x \tan x$ and $g_2(x)=-x \cot x$.}
\end{figure}
\begin{figure}
\begin{center}
\includegraphics[scale=0.4]{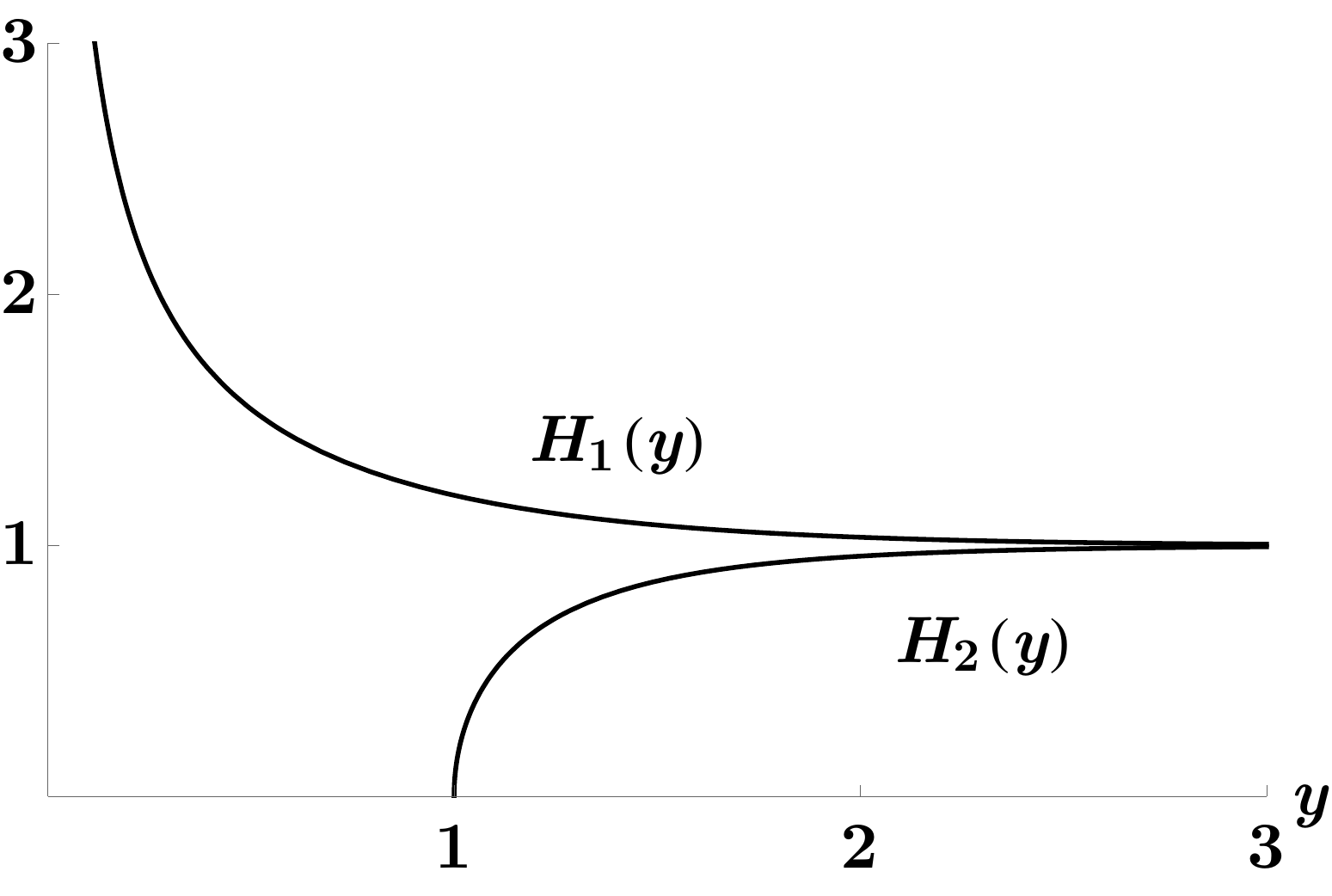}
\hspace{1cm}
\includegraphics[scale=0.4]{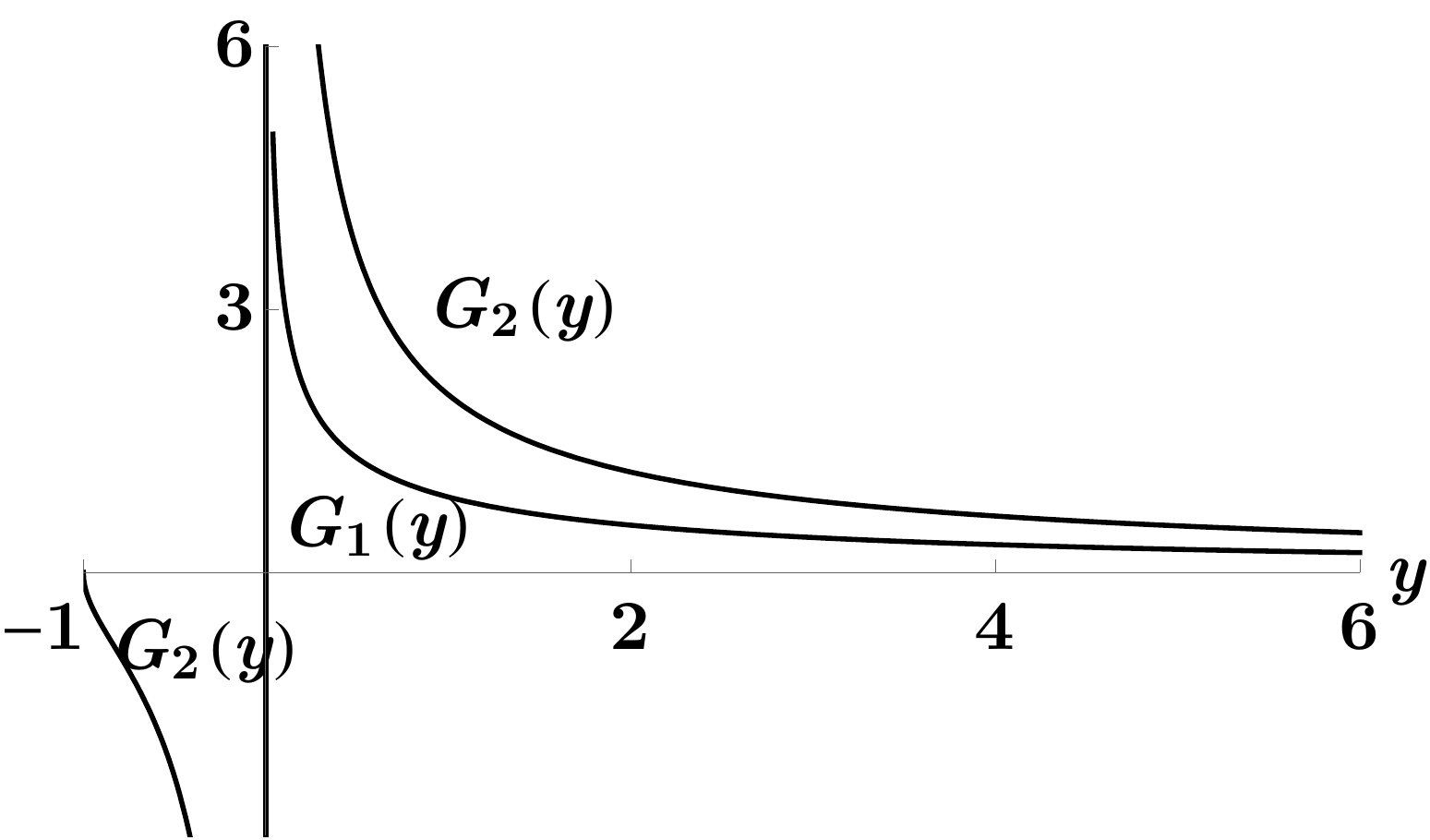}
\end{center}
\caption{\label{G1G2H1H2} The functions $\ho(y)=h_1^{-1}(y)/y$ and $\htwo(y)=h_2^{-1}(y)/y$, and $\go(y)=g_1^{-1}(y)/y$ and $\gt(y)=g_2^{-1}(y)/y$.}
\end{figure}
Define four more functions, shown in \autoref{G1G2H1H2}, by 
\[
\go(y) = \frac{g_1^{-1}(y)}{y} , \quad \gt(y) = \frac{g_2^{-1}(y)}{y} , \quad \ho(y) = \frac{h_1^{-1}(y)}{y} , \quad \htwo(y) = \frac{h_2^{-1}(y)}{y} , \quad 
\]
for $y$-values in the ranges of $g_1,g_2,h_1,h_2$ respectively. Three of them are strictly decreasing, while $\htwo$ is strictly increasing, as we now justify. 
\begin{lemma}[Monotonicity]\label{monotonicity}\ 
\begin{itemize}
\item[(i)] $\go^\prime(y)<0$ for all $y>0$.
\item[(ii)] $\gt^\prime(y)<0$ for all $y> -1, y \neq 0$.
\item[(iii)] $\ho^\prime(y)<0$ for all $y>0$.
\item[(iv)] $\htwo^\prime(y)>0$ for all $y>1$.
\end{itemize}
\end{lemma}
\begin{proof}
Given any strictly increasing function $h$ with $h^\prime>0$, we may write $x=h^{-1}(y)$ and differentiate the function $\h(y)=h^{-1}(y)/y=x/h(x)$ to obtain 
\begin{equation} \label{Hprime}
\h^\prime(y) = \frac{1}{h^\prime(x)} \left( \frac{x}{h(x)} \right)^{\! \prime} ,
\end{equation}
where the derivative on the left is taken with respect to $y$ and on the right with respect to $x$. Applying this derivative formula to the four functions in the lemma gives the following observations. 

$(x/g_1(x))^\prime = (\cot x)^\prime<0$, so that $\go^\prime < 0$.

$(x/g_2(x))^\prime = -(\tan x)^\prime<0$, so that $\gt^\prime < 0$. (Here $x \neq \pi/2$, since $y \neq 0$.)

$(x/h_1(x))^\prime = (\coth x)^\prime<0$, so that $\ho^\prime < 0$.

$(x/h_2(x))^\prime = (\tanh x)^\prime>0$, so that $\htwo^\prime > 0$.
\end{proof}
We proceed to develop concavity and derivative properties of the eight functions. 
\begin{lemma}[Concavity of the inverse squared]\label{concaveinverse}\ 
\begin{itemize}
\item[(i)] $(g_1^{-1}(y)^2)^{\prime \prime}<0$ for all $y>0$. 
\item[(ii)] $(g_2^{-1}(y)^2)^{\prime \prime}<0$ for all $y> -1$. 
\item[(iii)] $(h_1^{-1}(y)^2)^{\prime \prime}>0$ for all $y>0$. 
\item[(iv)] $(h_2^{-1}(y)^2)^{\prime \prime}>0$ for all $y>1$. 
\end{itemize}
\end{lemma}
\begin{proof}
Writing $y=h(x)$, we find
\[
\frac{1}{2} \frac{d^2\ }{dy^2} \left( h^{-1}(y)^2 \right) = \frac{1}{2} \frac{d^2(x^2)}{dy^2} = \frac{h^\prime(x) - x h^{\prime \prime}(x)}{h^\prime(x)^3} .
\]
Replacing $h$ on the right side with $g_1,g_2,h_1,h_2$ respectively gives the following expressions, where for $g_2$ we split the interval in two pieces:
\begin{align}
\text{(i)} \hspace{2cm} & -\frac{4\cos^4 x}{(2x+\sin 2x)^3} (2x - \sin 2x + 4 x^2 \tan x) < 0 , \qquad x \in (0,\pi/2) , \label{inverse-i} \\
\text{(ii)} \hspace{2cm} & -\frac{4\sin^3 x \cos x}{(2x-\sin 2x)^3} (2x \tan x  + 2\sin^2 x - 4 x^2) < 0 , \qquad x \in (0,\pi/2) , \label{inverse-iia} \\
\text{(ii)} \hspace{2cm} & -\frac{4\sin^4 x}{(2x-\sin 2x)^3} (2x + \sin 2x - 4 x^2 \cot x) < 0 , \qquad x \in [\pi/2,\pi) , \label{inverse-iib} \\
\text{(iii)} \hspace{2cm} & \frac{4\cosh^4 x}{(\sinh 2x +  2x)^3} (\sinh 2x - 2x + 4 x^2 \tanh x) > 0 , \qquad x \in (0,\infty) , \label{inverse-iii} \\
\text{(iv)} \hspace{2cm} & \frac{4\sinh^3 x \cosh x}{(\sinh 2x - 2x)^3} (2\sinh^2 x + 2x \tanh x - 4 x^2) > 0 , \qquad x \in (0,\infty) , \label{inverse-iv} 
\end{align}
noting that the inequalities \eqref{inverse-i}, \eqref{inverse-iib} and \eqref{inverse-iii} use only that $\pm \sin 2x < 2x < \sinh 2x$, while inequalities \eqref{inverse-iia} and \eqref{inverse-iv} are proved as follows. 

To show negativity in \eqref{inverse-iia}, observe that $2x \tan x  + 2\sin^2 x - 4 x^2$ is positive for all $x \in (0,\pi/2)$ because for small $x$ it behaves like $O(x^2)$ and its second derivative is positive on the whole interval: 
\[
(2x \tan x + 2\sin^2 x - 4 x^2)^{\prime \prime} = \tan x \sec^2 x (4x-\sin 4x) > 0 .
\]
The argument for \eqref{inverse-iv} is analogous: $2\sinh^2 x + 2x \tanh x - 4 x^2$ behaves like  $O(x^2)$ for small $x$ and has positive second derivative everywhere since 
\[
(2\sinh^2 x + 2x \tanh x - 4 x^2)^{\prime \prime} = \tanh x \sech^2 x (\sinh 4x - 4x) > 0 .
\]
\end{proof}

\subsection*{Definition of $\alpha_\pm$}
As needed for the statement of \autoref{lambda2L}, let $\alpha_+ \simeq 33.2054$ be the root of
\begin{equation}
g_1^{-1}(\alpha/8)^2 + g_2^{-1}(\alpha/8)^2 = \alpha/4 , \label{alphaplus} 
\end{equation}
and $\alpha_- \simeq -9.3885$ be the root of
\begin{equation}
h_1^{-1}(|\alpha|/8)^2 + h_2^{-1}(|\alpha|/8)^2 = |\alpha|/4 . \label{alphaneg}
\end{equation}
That these roots exist and are unique can be seen as follows. At $\alpha=0$, the left side of \eqref{alphaplus} is $0^2+(\pi/2)^2$ while the right side is $0$, and so the left side is larger. As $\alpha \to \infty$ the left side of \eqref{alphaplus} approaches $(\pi/2)^2+\pi^2$ while the right side approaches $\infty$, and so the right side is larger. Hence \eqref{alphaplus} has a root at some $\alpha_+>0$, and the root is unique because  \autoref{concaveinverse} implies that $(g_1^{-1})^2+(g_2^{-1})^2$ is strictly concave. Similarly, at $\alpha=-8$, the left side of \eqref{alphaneg} is $h_1^{-1}(1)^2+0^2$ while the right side is $2$, and so the left side is smaller (one computes that $h_1(\sqrt{2})>1$ and so $h_1^{-1}(1)<\sqrt{2})$. As $\alpha \to -\infty$ the left side of \eqref{alphaneg} grows quadratically with $\alpha$, while the right side grows only linearly, and so the left side is larger. Hence \eqref{alphaneg} has a root for some $\alpha_-<-8$, and the root is unique because  \autoref{concaveinverse} implies that $(h_1^{-1})^2+(h_2^{-1})^2$ is strictly convex. The  roots $\alpha_+$ and $\alpha_-$ can then be approximated numerically, to find the values stated above. \qed

\begin{lemma}[Bounds on the inverse]\label{inversebounds}
For all $y>0$,  
\[
g_1^{-1}(y)^2 > y - y^2  \qquad \text{and} \qquad h_1^{-1}(y)^2 < y + y^2 .
\]
\end{lemma}
\begin{proof}
Putting $y=g_1(x)=x \tan x$ into the first inequality shows it is equivalent to $x(1+\tan^2 x) > \tan x$, which reduces to $x > \sin x \cos x$. For the second inequality, putting $y=h_1(x)=x \tanh x$ reduces it to $x(1-\tanh^2 x) < \tanh x$ and hence to $x < \sinh x \cosh x$.
\end{proof}
\begin{lemma}[Asymptotics of the inverse]\label{inverseasymptotics} As $y \to \infty$, 
\[
h_1^{-1}(y) = y \left( 1 + O(e^{-2y}) \right) \qquad \text{and} \qquad h_2^{-1}(y) = y \left( 1 + O(e^{-2y}) \right) .
\] 
\end{lemma}
\begin{proof}
If $y=h_1(x)=x \tanh x \leq x$ then $x=y \coth x = y \left( 1 + O(e^{-2x}) \right) = y \left( 1 + O(e^{-2y}) \right)$. 

If $y=h_2(x)=x \coth x = x \left( 1 + O(e^{-2x}) \right)$ then $y-x \to 0$ as $x \to \infty$, and so $x=y \left( 1 + O(e^{-2y}) \right)$. 
\end{proof}
\begin{lemma}[Derivative comparison]\label{derivcomparison}
For all $y>0$, we have 
\begin{equation} \label{Hderiv}
\go^\prime(y) > \gt^\prime(y) 
\end{equation}
and 
\begin{equation} \label{inversederiv}
\frac{d\ }{dy} \big( g_2^{-1}(y)^2 - g_1^{-1}(y)^2 \big) > 0 .
\end{equation}
\end{lemma}
\begin{proof}
By the derivative formula \eqref{Hprime}, the conclusion $\go^\prime(y) > \gt^\prime(y)$ is equivalent to 
\[
\left. \frac{1}{g_1^\prime(x)} \left( \frac{x}{g_1(x)} \right)^{\! \prime} \right|_{x_1=g_1^{-1}(y)}
> 
\left. \frac{1}{g_2^\prime(x)} \left( \frac{x}{g_2(x)} \right)^{\! \prime} \right|_{x_2=g_2^{-1}(y)} ,
\]
which evaluates to 
\[
\frac{2 \cot^2 x_1}{2x_1 + \sin 2x_1} < \frac{2 \tan^2 x_2}{2x_2 - \sin 2x_2} .
\]
Multiply on the left by $y^2 = g_1(x_1)^2 = x_1^2 \tan^2 x_1$ and on the right by $y^2 = g_2(x_2)^2 = x_2^2 \cot^2 x_2$, and hence obtain that the desired inequality is equivalent to 
\[
\frac{2 x_1^2}{2x_1 + \sin 2x_1} < \frac{2 x_2^2}{2x_2 - \sin 2x_2} ,
\]
or 
\[
\frac{2x_1}{1 + \sinc 2x_1} < \frac{2x_2}{1 - \sinc 2x_2} .
\]
Recall that $2x_1 \in (0,\pi)$ and $2x_2 \in (\pi,2\pi)$, so that $\sinc 2x_1$ is positive and $\sinc 2x_2$ is negative. Thus it suffices to show that the function $x/(1+|\sinc x|)$ is strictly increasing when $x \in (0,2\pi)$. This last fact is easily verified, since
\[
\left( \frac{x}{1+\sinc x} \right)^{\! \prime} = \frac{x(1-\cos x)+ 2\sin x}{x(1+\sinc x)^2} > 0
\]
for $x \in (0,\pi)$, and one can argue similarly when $x \in (\pi,2\pi)$.

Next, formula \eqref{inversederiv} says
\[
g_2^{-1}(y) (g_2^{-1})^\prime(y) > g_1^{-1}(y) (g_1^{-1})^\prime(y) .
\]
Again let $x_1=g_1^{-1}(y) \in (0,\pi/2)$ and $x_2=g_2^{-1}(y) \in (\pi/2,\pi)$. Since $g_1^\prime>0$ and $g_2^\prime>0$, the preceding inequality is equivalent to 
\[
\frac{g_1^\prime(x_1)}{x_1} > \frac{g_2^\prime(x_2)}{x_2} .
\]
Substituting the definitions of $g_1$ and $g_2$ and using the identities $\sec^2=\tan^2+1$ and $\csc^2=\cot^2+1$ now reduces the inequality to 
\[
\frac{y+y^2}{x_1^2} + 1 > \frac{y+y^2}{x_2^2} + 1,
\]
which certainly holds true since $x_1<\pi/2<x_2$.
\end{proof}

\begin{lemma}[More derivative comparison]\label{derivcomparisonneg}\ 

(i) 
\begin{equation} \label{inversederivneg1}
\frac{d\ }{dy} \big( g_2^{-1}(y)^2 + h_1^{-1}(-y)^2 \big) > 0 , \qquad -1<y<0 .
\end{equation}
(ii) 
\begin{equation} \label{inversederivneg2}
\frac{d\ }{dy} \big( -h_2^{-1}(y)^2 + h_1^{-1}(y)^2 \big) < 0 , \qquad 1<y<\infty .
\end{equation}
\end{lemma}
\begin{proof}
(i) Let $x_2=g_2^{-1}(y)$ and $x_1=h_1^{-1}(-y)$. 
Formula \eqref{inversederivneg1} holds if and only if
\[
g_2^{-1}(y) (g_2^{-1})^\prime(y) > h_1^{-1}(-y) (h_1^{-1})^\prime(-y) .
\]
Since $g_2^\prime>0$ and $h_1^\prime>0$, the inequality is equivalent to 
\[
\frac{h_1^\prime(x_1)}{x_1} > \frac{g_2^\prime(x_2)}{x_2} .
\]
Substituting the definitions of $h_1$ and $g_2$ and using the identities $\sech^2=1-\tanh^2$ and $\csc^2=\cot^2+1$, and recalling $y=-h_1(x_1)=g_2(x_2)$, the inequality simplifies to 
\[
- \frac{y+y^2}{x_1^2} + 1 > \frac{y+y^2}{x_2^2} + 1,
\]
which is true since $-1<y<0$ implies $y+y^2<0$.

(ii) Let $x_1=h_1^{-1}(y)$ and $x_2=h_2^{-1}(y)$. 
Formula \eqref{inversederivneg2} holds if and only if
\[
h_2^{-1}(y) (h_2^{-1})^\prime(y) > h_1^{-1}(y) (h_1^{-1})^\prime(y) .
\]
Since $h_1^\prime>0$ and $h_2^\prime>0$, the inequality is equivalent to 
\[
\frac{h_1^\prime(x_1)}{x_1} > \frac{h_2^\prime(x_2)}{x_2} .
\]
Substituting the definitions of $h_1$ and $h_2$ and using the identities $\sech^2=1-\tanh^2$ and $\csch^2=\coth^2-1$, the inequality simplifies to 
\[
\frac{y-y^2}{x_1^2} + 1 > \frac{y-y^2}{x_2^2} + 1 .
\]
Note $1<y<\infty$ implies $y-y^2<0$, and so the task is to show $x_2 < x_1$, or in other words
\begin{equation}\label{h1h2}
h_2^{-1}(y) < h_1^{-1}(y) , \qquad y > 1 .
\end{equation}
This inequality holds since $h_1$ and $h_2$ are increasing and $h_1(x)<h_2(x)$ for all $x$.
\end{proof}
\begin{lemma}[More monotonicity]\label{specific}
(i) The function
\[
f_1(x) = \frac{1}{h_1(x)} - \frac{x h_1^\prime(x)}{2h_1(x)^2}
\]
is strictly decreasing, with $f_1(x) \to 1/3$ as $x \to 0+$ and $f_1(x) \to 0$ as $x \to \infty$. Further, $h_1(f_1^{-1}(w))<1/2w$ for all $w \in (0,1/3)$. 

(ii) Similarly
\[
f_2(x) = \frac{1}{h_2(x)} - \frac{x h_2^\prime(x)}{2h_2(x)^2}
\]
is strictly decreasing, with $f_2(x) \to 1$ as $x \to 0+$ and $f_2(x) \to 0$ as $x \to \infty$. Further, $f_1(x)<f_2(x)$ for all $x>0$, and $h_2(f_2^{-1}(w))<1/w$ for all $w \in (0,1)$, and $h_1(f_1^{-1}(w))+h_2(f_2^{-1}(w))<1/w$ for all $w \in (0,1/3)$. 

\end{lemma}
\begin{proof}
(i) Direct calculation with $h_1(x)=x \tanh x$ gives that 
\[
f_1(x) = \frac{\coth x - x \csch^2 x}{2x} 
\]
and so $f_1(0+)=1/3$ by elementary series expansions, with $\lim_{x \to \infty} f_1(x)=0$. One computes 
\[
f_1^\prime(x) = \frac{4x^2 \coth x - \sinh 2x - 2x}{4x^2 \sinh^2 x} .
\]
We will show the numerator is negative, so that $f_1$ is strictly decreasing. 

Notice $4x^2 \coth x - \sinh 2x - 2x = 0$ at $x=0$. Thus it suffices to show the first derivative is negative for all $x>0$, which is clear because
\[
(4x^2 \coth x - \sinh 2x - 2x)^\prime
= -4 (\cosh x - x \csch x)^2 <  0 .
\]

For the second claim in part (i), rearrange the definition of $f_1$ to get that 
\begin{equation} \label{h1f1}
h_1(x) = \frac{1}{f_1(x)} \left( 1 - \frac{x h_1^\prime(x)}{2h_1(x)} \right) = \frac{1}{2f_1(x)} \left( 1 - \frac{1}{\sinch 2x} \right) 
\end{equation}
by substituting $h_1(x)=x \tanh x$. Hence $h_1(x) < 1/2f_1(x)$, and so $h_1(f_1^{-1}(w)) < 1/2w$ for all $w$ in the range of $f_1$, which is $(0,1/3)$. 

(ii) Substituting $h_2(x)=x \coth x$ into the definition of $f_2$ gives that 
\[
f_2(x) = \frac{\tanh x + x \sech^2 x}{2x} ,
\]
from which one evaluates the limit as $f_2(0+)=1$, and obviously $\lim_{x \to \infty} f_2(x)=0$. Differentiating, we find
\[
f_2^\prime(x) = - \frac{4x^2 \tanh x + \sinh 2x - 2x}{4x^2 \cosh^2 x} < 0 ,
\]
so that $f_2$ is strictly decreasing. 

Further, the inequality $f_1(x)<f_2(x)$ holds when $x>0$ because it is equivalent to $\tanh 2x < 2x$, by manipulating the formulas above for $f_1(x)$ and $f_2(x)$. 

For the final claims in part (ii) of the lemma, rearrange the definition of $f_2$ so as to express $h_2$ in terms of $f_2$:
\begin{equation} \label{h2f2}
h_2(x) = \frac{1}{f_2(x)} \left( 1 - \frac{x h_2^\prime(x)}{2h_2(x)} \right) = \frac{1}{2f_2(x)} \left( 1 + \frac{1}{\sinch 2x} \right) 
\end{equation}
by substituting $h_2(x)=x \coth x$. The middle part of formula \eqref{h2f2} implies $h_2(x) < 1/f_2(x)$, and so $h_2(f_2^{-1}(w)) < 1/w$ for all $w$ in the range of $f_2$, that is, for all $w \in (0,1)$. 

By adding formulas \eqref{h1f1} and \eqref{h2f2} we deduce
\[
h_1(f_1^{-1}(w))+h_2(f_2^{-1}(w)) = \frac{1}{2w} \left( 2 - \frac{1}{\sinch 2f_1^{-1}(w)} + \frac{1}{\sinch 2f_2^{-1}(w)} \right) 
\]
for all $w \in (0,1/3)$. This last expression is less than $1/w$ since $f_1^{-1}(w) < f_2^{-1}(w)$, using here that $f_1$ and $f_2$ are decreasing with $f_1<f_2$. 
\end{proof}

Next we examine situations where $\h(y)$ is strictly convex with respect to $\log y$. 
\begin{lemma}[Convexity with respect to $\log y$]\label{convexity}
The functions $\go(y)$ and $\ho(y)$ are strictly convex with respect to $\log y$, with 
\[
\frac{d^2\ }{dz^2} \go(e^z) > 0 \quad \text{and} \quad \frac{d^2\ }{dz^2} \ho(e^z) > 0 , \qquad z \in \R .
\]
\end{lemma}
\begin{proof}
A straightforward calculation shows 
\begin{equation}\label{secondderiv}
\frac{d^2\ }{dz^2} \h(e^z) = - \frac{h(x) h^{\prime \prime}(x)}{h^\prime(x)^3} - \frac{1}{h^\prime(x)} + \frac{x}{h(x)} ,
\end{equation}
whenever $H(y)=h^{-1}(y)/y$ and $y=e^z=h(x)$. The task is to show the right side of \eqref{secondderiv} is positive when $h$ is replaced by $g_1$ and also when it is replaced by $h_1$. Note $g_1^\prime>0$ and $h_1^\prime>0$, and so the denominators are positive in every case. 

When $h=g_1$, the right side of \eqref{secondderiv} equals
\[
\frac{8x \cot x}{(2x+\sin 2x)^3} \big(x^2-\sin^2 x \cos^2 x +2x \cos^3 x \sin x \big) ,
\]
which is positive because $x^2 \geq \sin^2 x$.

When $h=h_1$, the right side of \eqref{secondderiv} evaluates to 
\[
\frac{8x \coth x}{(2x+\sinh 2x)^3} f(x) 
\]
where $f(x)=x^2-\sinh^2 x \cosh^2 x + 2x \cosh^3 x \sinh x$. Notice $f(0)=0,f^\prime(0)=0$ and 
\[
f^{\prime \prime}(x) = 4 \cosh^2 x + 4 x \sinh x \cosh x (3\sinh^2 x + 5\cosh^2 x ) > 0 ,
\]
from which it follows that $f(x)>0$ for all $x>0$, completing the proof. 
\end{proof}
Convexity or concavity of functions in the form $y(1-y)\h(cy)^2$ will be important for our arguments too. 
\begin{lemma}[Convexity with $\go$ and $\gt$]\label{convexityyy} Fix $c>0$. 

(i) $y(1-y)\go(cy)^2$ is a strictly convex function of $y>0$. 

(ii) $y(1-y)\gt(cy)^2$ is a strictly convex function of $y>0$. 

(iii) $y(1-y)\gt(-cy)^2$ is a strictly decreasing function of $0<y<\min(1,1/c)$. 
\end{lemma}
\begin{proof}
Part (i). Equivalently, we show strict convexity of $y(1-y/c)\go(y)^2$ for $y>0$. Direct differentiation gives
\[
\frac{1}{2} \frac{d^2\ }{dy^2} \big( y(1-y/c)\go(y)^2 \big) 
=  \frac{x^2}{g_1(x)^3} - \frac{2 x}{g_1(x)^2 g_1^\prime(x)} + \frac{1 - g_1(x)/c}{g_1(x) g_1^\prime(x)^2} - \frac{x (1- g_1(x)/c) g_1^{\prime \prime}(x)}{g_1(x) g_1^\prime(x)^3}
\]
where $y=g_1(x)$. The right side of this formula can be rewritten as
\begin{equation} \label{bigrightside}
\frac{x^2}{g_1(x)^3} - \frac{2 x}{g_1(x)^2 g_1^\prime(x)} + \frac{1}{g_1(x) g_1^\prime(x)^2} - \frac{x g_1^{\prime \prime}(x)}{g_1(x) g_1^\prime(x)^3} - \frac{1}{c} \left( \frac{g_1^\prime(x) - x g_1^{\prime \prime}(x)}{g_1^\prime(x)^3} \right) . 
\end{equation}
The parenthetical term in \eqref{bigrightside} is negative by \eqref{inverse-i}.

The other terms in \eqref{bigrightside} equal 
\[
\frac{2\cot^3 x}{(2x + \sin 2x)^3} f(x)
\]
where $f(x) = -1 + 4x^2 + \cos 4x + x \sin 4x$. To complete the proof we will show $f(x)$ is positive when $x>0$. Obviously $f(0)=0$, and so it suffices to show $f^\prime(x)>0$. One computes
\[
f^\prime(x) = 4x ( 2 + \cos 4x - 3 \sinc 4x) ,
\]
which is positive whenever $x>3/4$ because 
\[
2 + \cos 4x - 3 \sinc 4x \geq 1 - \frac{3}{4x} > 0.
\]
Further, power series expansions yield
\[
f^\prime(x) = 8x \sum_{k=2}^\infty (-1)^k \frac{k-1}{(2k+1)!} (4x)^{2k} .
\]
When $0 < x \leq 1$, the terms of this alternating series decrease in magnitude as $k$ increases, so that $f^\prime(x)>0$, completing this part of the proof. 

Part (ii). In formula \eqref{bigrightside} we replace $g_1(x)$ by $g_2(x)=-x \cot x$. The parenthetical term in \eqref{bigrightside} is then negative by \eqref{inverse-iib}, noting 
$x \in (\pi/2,\pi)$ because $g_2(x)=y>0$ in this part of the lemma. The other terms in \eqref{bigrightside} equal
\[
-\frac{2\tan^3 x}{(2x - \sin 2 x)^3} f(x)
\]
where $f(x)$ is the function used in part (i) of the proof. Since $f$ is positive, we see the last expression is positive when $\pi/2<x<\pi$, as we wanted to show.

Part (iii). Rescale by $c$ and consider $y(1-y/c)\gt(-y)^2$ for $0<y<\min(c,1)$. Differentiating gives
\[
\frac{d\ }{dy} \left( y(1-y/c)\gt(-y)^2 \right) = \frac{2x}{g_2^\prime(x)} \left( \frac{1}{c}  + \frac{1}{g_2(x)} - \frac{x g_2^\prime(x)}{2g_2(x)^2} \right) 
\]
where $y=-g_2(x)$. The right side is negative since $g_2^\prime(x)>0$ and $0<-g_2(x)<c$. 
\end{proof}
\begin{lemma}[Concavity with $\ho$]\label{concavityH1} If $0 < c \leq 3$ then $y(1-y)\ho(cy)^2$ is strictly decreasing and strictly concave for $y \in (0,1)$.

If $c>3$ then a number $y_1(c) \in (0,1/2)$ exists such that $y(1-y)\ho(cy)^2$ is:
\begin{itemize}
\item strictly increasing for $y \in \big(0, y_1(c)\big)$, 
\item strictly decreasing for $y \in \big(y_1(c),1\big)$, 
\item strictly concave for $y \in \big(y_1(c),1\big)$. 
\end{itemize}
\end{lemma}
To unify the two parts of this lemma, one simply defines $y_1(c)=0$ when $0 < c \leq 3$. 
\begin{proof}
After rescaling, we consider the function $K(y) = y(1-y/c)\ho(y)^2$ and show that if $0 < c \leq 3$ then $K(y)$ is strictly decreasing and strictly concave for $y \in (0,c)$, while if $c>3$ then a number $y_c \in (0,c/2)$ exists such that $K(y)$ is 
\begin{itemize}
\item strictly increasing for $y \in (0, y_c)$, 
\item strictly decreasing for $y \in (y_c,c)$, 
\item strictly concave for $y \in (y_c,c)$. 
\end{itemize}
The values $y_c$ and $y(c)$ are related by $y_c = c y_1(c)$. 

To begin with, differentiating the definition of $K$ yields
\[
K^\prime(y) = \frac{2x}{h_1^\prime(x)} \left( f_1(x) - \frac{1}{c} \right) 
\]
where $y=h_1(x)$ and 
\[
f_1(x) = \frac{1}{h_1(x)} - \frac{x h_1^\prime(x)}{2h_1(x)^2} , \qquad x>0 .
\]
\autoref{specific} says this function $f_1$ is strictly decreasing and has limiting values $f_1(0+)=1/3$ and $\lim_{x \to \infty} f_1(x)=0$, so that $0<f_1(x)<1/3$ for all $x>0$.

If $c \leq 3$ then $f_1(x)<1/c$ for all $x>0$, and so $K^\prime(y)<0$ for all $y>0$. Let $y_c=0$ in this case. 

If $c>3$ then $f_1(x_c) = 1/c$ for a unique number $x_c>0$. Letting $y_c=h_1(x_c)$, we have that $K^\prime(y)>0$ when $x<x_c$, that is, when $y<y_c$. Also, $K^\prime(y)<0$ when $y>y_c$. Further, $h_1(f_1^{-1}(1/c))<c/2$ by \autoref{specific}(i) (noting $1/c < 1/3$) and so $h_1(x_c) < c/2$, or $y_c<c/2$, as desired.

Concavity of $K(y)$ remains to be proved, when $y_c < y < c$, which is the interval on which $f_1(x)<1/c$. Formula \autoref{bigrightside} with $g_1$ replaced by $h_1$ shows that 
\[
\frac{1}{2} K^{\prime \prime}(y) = 
\frac{x^2}{h_1(x)^3} - \frac{2 x}{h_1(x)^2 h_1^\prime(x)} + \frac{1}{h_1(x) h_1^\prime(x)^2} - \frac{x h_1^{\prime \prime}(x)}{h_1(x) h_1^\prime(x)^3} - \frac{1}{c} \left( \frac{h_1^\prime(x) - x h_1^{\prime \prime}(x)}{h_1^\prime(x)^3} \right) . 
\]
The parenthetical term is positive by \eqref{inverse-iii}. Thus we may replace $1/c$ with the smaller number $f_1(x)$, getting an upper bound 
\begin{align*}
\frac{1}{2} K^{\prime \prime}(y) 
& < \frac{x^2}{h_1(x)^3} - \frac{2 x}{h_1(x)^2 h_1^\prime(x)} + \frac{1}{h_1(x) h_1^\prime(x)^2} - \frac{x h_1^{\prime \prime}(x)}{h_1(x) h_1^\prime(x)^3} - f_1(x) \left( \frac{h_1^\prime(x) - x h_1^{\prime \prime}(x)}{h_1^\prime(x)^3} \right) \\
& = \frac{2\cosh^2 x \coth^3 x}{x(2x + \sinh 2x)^2} \, \tilde{f}(x)
\end{align*}
by substituting $h_1(x)=x \tanh x$, where 
\[
\tilde{f}(x) = 2x^2-\sinh^2 x - x \tanh x .
\]
To show $K^{\prime \prime}(y)<0$ we want $\tilde{f}(x)<0$ for all $x>0$. This is true, since $\tilde{f}(0)=0,\tilde{f}^\prime(0)=0$ and 
\[
\tilde{f}^{\prime \prime}(x) = \frac{\sinh x}{2\cosh^3 x} (4x - \sinh 4x) < 0.
\]
\end{proof}
\begin{lemma}[Concavity with $\htwo$]\label{concavityhtwo} 
Let $c>0$. The function $y(1-y)\htwo(cy)^2$ is strictly concave for $y>1/c$, and is 
\begin{itemize}
\item strictly increasing for $y \in \big(1/c, y_2(c)\big)$, 
\item strictly decreasing for $y \in \big(y_2(c),\infty\big)$, 
\end{itemize}
for some number $y_2(c) \geq 1/c$. Furthermore, $y_1(c)+y_2(c)<1$ whenever $c>1$, where the number $y_1(c)$ was constructed in \autoref{concavityH1}.
\end{lemma}
\begin{proof}
Rescaling by $c>0$, we want to show strict concavity for the function 
\[
K(y) = y(1-y/c)\htwo(y)^2 = \left( \frac{1}{y} - \frac{1}{c} \right) h_2^{-1}(y)^2 , \qquad y > 1 .
\]
The second derivative $K^{\prime \prime}(y)$ is given by formula \eqref{bigrightside}, except replacing $g_1$ with $h_2$ there. The parenthetical term in this new version of \eqref{bigrightside} is positive by estimate \eqref{inverse-iv}, and so the factor of $-1/c$ in \eqref{bigrightside} makes that term negative.

The other terms in \eqref{bigrightside}, after $g_1(x)$ is replaced throughout by $h_2(x)=x \coth x$, evaulate to 
\[
-\frac{2\tanh^3 x}{(\sinh 2x - 2x)^3} f(x)
\]
where $f(x) = 1 + 4x^2 - \cosh 4x + x \sinh 4x$. To finish the concavity proof we show $f(x)>0$ for all $x>0$, so that $K^{\prime \prime}(y)<0$. Obviously $f(0)=0$, and the derivative is positive since 
\[
f^\prime(x) = 4x ( 2 + \cosh 4x - 3 \sinch 4x) = \sum_{k=2}^\infty \frac{(4x)^{2k+1}}{(2k)!} \left( 1 - \frac{3}{2k+1} \right) > 0
\]
by the usual power series expansions. 

For the monotonicity assertions in the lemma, we want a number $y_c = cy_2(c) \geq 1$ such that $K(y)$ is 
\begin{itemize}
\item strictly increasing for $y \in (1, y_c)$, 
\item strictly decreasing for $y \in (y_c,\infty)$. 
\end{itemize}
To get these results, differentiate the definition of $K$ to find 
\[
K^\prime(y) = \frac{2x}{h_2^\prime(x)} \left( f_2(x) - \frac{1}{c} \right) 
\]
where $y=h_2(x)$ and 
\[
f_2(x) = \frac{1}{h_2(x)} - \frac{x h_2^\prime(x)}{2h_2(x)^2} , \qquad x>0 .
\]
\autoref{specific}(ii) says that this function $f_2$ decreases strictly in value from $1$ to $0$ as $x$ increases from $0$ to $\infty$. 

If $0 < c \leq 1$ then $f_2(x)<1/c$ for all $x>0$, and choosing $y_c=1$ gives $K^\prime(y)<0$ for all $y>y_c$. If $c>1$ then $f_2(x_c) = 1/c$ for a unique number $x_c>0$. Letting $y_c=h_2(x_c)$, we have that $K^\prime(y)>0$ when $1<y<y_c$ (or $0<x<x_c$), while $K^\prime(y)<0$ when $y>y_c$ (or $x>x_c$). This proves the monotonicity claims in the lemma. 

Finally, we want to prove $y_1(c)+y_2(c)<1$ when $c>1$. When $1 < c \leq 3$ we recall $y_1(c)=0$ by the comment after \autoref{concavityH1}, and so the task in this range is to show $y_2(c)<1$. By definition $y_2(c)=y_c/c=h_2(x_c)/c$, and so $y_2(c)<1$ if and only if $h_2(f_2^{-1}(1/c))<c$. This last inequality follows from \autoref{specific}(ii) with $w=1/c \in (0,1)$. Next, when $c > 3$ we use the definition of $y_1(c)$ in \autoref{concavityH1} to show that $y_1(c)+y_2(c)<1$ if and only if $h_1(f_1^{-1}(1/c))+h_2(f_2^{-1}(1/c))<c$, which then follows from \autoref{specific}(ii) with $w=1/c \in (0,1/3)$. 
\end{proof}

\section{\bf The first and second Robin eigenvalues of an interval}
\label{identifyinginterval}

The Rayleigh quotient 
\[
\frac{\int_{-t}^t (u^\prime)^2 \, dx + \alpha \left( u(t)^2 + u(-t)^2 \right) }{\int_{-t}^t u^2 \, dx} 
\]
for the interval 
\[
\mathcal{I}(t)=(-t,t)
\]
generates the eigenvalue equation $-u^{\prime \prime} = \lambda u$ with boundary condition $\partial u / \partial \nu + \alpha u = 0$ at $x=\pm t$. Our results depend on understanding how the first and second Robin eigenvalues of the interval depend on the half-length $t>0$ and Robin parameter $\alpha \in \R$. 

The following lemmas each state two eigenvalue formulas. The first formula involves $g_1,g_2,h_1,h_2$ (which were defined in \autoref{notation}), and is useful when the half-length $t$ is fixed and the Robin parameter $\alpha$ is varying. The second formula involves $\go,\gt,\ho,\htwo$ (also defined in \autoref{notation}), and is useful when $\alpha$ is fixed and $t$ is varying. 
\begin{lemma} \label{firsteigen}
\[
\lambda_1(\mathcal{I}(t);\alpha) = 
\begin{cases}
g_1^{-1}(\alpha t)^2/t^2 \\
0 \\
-h_1^{-1}(-\alpha t)^2/t^2 
\end{cases}
=
\begin{cases}
\alpha^2 \go(\alpha t)^2 & \text{if $\alpha>0$,} \\
0 & \text{if $\alpha=0$,} \\
-\alpha^2 \ho(-\alpha t)^2 & \text{if $\alpha<0$.}
\end{cases}
\]
This first eigenvalue is a strictly increasing function of $\alpha \in \R$. As $\alpha \to -\infty$ it equals $-\alpha^2 \left(1+O(e^{2\alpha t}) \right)$, and as $\alpha \to \infty$ it converges to $(\pi/2t)^2$.
\end{lemma}
\noindent A more precise asymptotic as $\alpha \to -\infty$ was given by \text{Antunes \emph{et al.}\ \cite[Proposition 1]{AFK17}.}
\begin{lemma} \label{secondeigen}
\[
\lambda_2(\mathcal{I}(t);\alpha) = 
\begin{cases}
g_2^{-1}(\alpha t)^2/t^2 \\
0 \\
-h_2^{-1}(-\alpha t)^2/t^2 
\end{cases}
=
\begin{cases}
\alpha^2 \gt(\alpha t)^2 & \text{if $\alpha> -1/t$,} \\
0 & \text{if $\alpha= -1/t$,} \\
-\alpha^2 \htwo(-\alpha t)^2 & \text{if $\alpha< -1/t$.}
\end{cases}
\]
This second eigenvalue is a strictly increasing function of $\alpha \in \R$. As $\alpha \to -\infty$ it equals $-\alpha^2 \left(1+ O(e^{2\alpha t}) \right)$, and as $\alpha \to \infty$ it converges to $(\pi/t)^2$.
\end{lemma}
When $\alpha=0$ the upper right formula in \autoref{secondeigen} is not well defined, because $\gt(0)$ is undefined (its denominator being zero). The upper left formula $g_2^{-1}(0)^2/t^2=(\pi/2t)^2$ still gives the correct value for $\lambda_2(\mathcal{I}(t);0)$.  

\autoref{firstsixinterval} plots the first six eigenvalues as functions of $\alpha$ for $t=1$, that is, for the interval $\mathcal{I}(1)=(-1,1)$. Formulas for these eigenvalue curves could be obtained from the proofs of the lemmas below. 
\begin{figure}
\begin{center}
\includegraphics[scale=0.5]{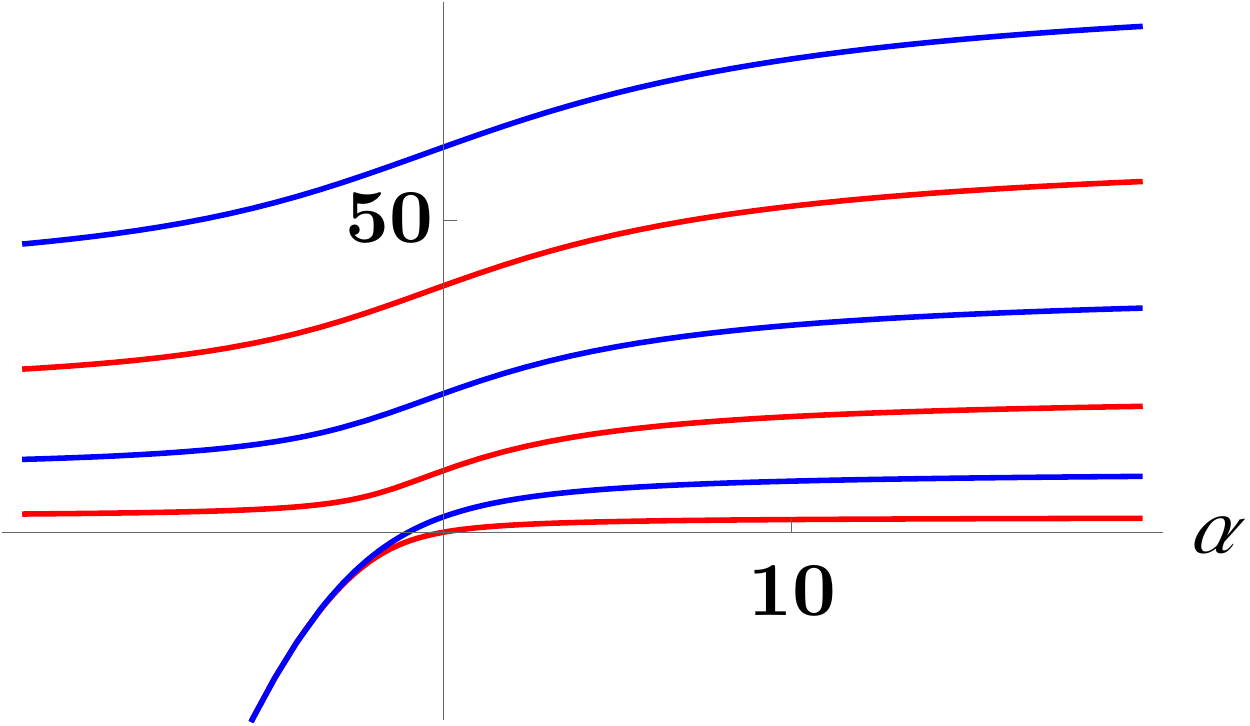}
\end{center}
\caption{\label{firstsixinterval} The first six eigenvalues $\lambda_k(\mathcal{I}(1);\alpha)$ for $k=1,\dots,6$ of the interval $\mathcal{I}(1)=(-1,1)$, plotted as functions of the Robin parameter $\alpha$. The eigenvalues come in pairs, corresponding to even and odd eigenfunctions. The even eigenvalue is always lower than the odd one.}
\end{figure}

\begin{proof}[Proof of \autoref{firsteigen} and \autoref{secondeigen}]
The Robin spectrum of the interval is known, of course, but the proofs and notations vary in clarity and notation, and many authors examine only $\alpha>0$. So it seems helpful to include a proof here, using our notation. 

Fix $t>0$. By symmetry of the interval, we may assume each eigenfunction is either even or odd. Thus the eigenvalue problem is  
\[
\begin{split}
u^{\prime \prime} +  \lambda u & =0 , \qquad 0<x<t,\\
u^\prime(t) + \alpha u(t) & = 0 .
\end{split}
\]

(i) First suppose $\lambda<0$, and write $\lambda= -\rho^2$ where $\rho>0$, so that the eigenfunction equation says $u^{\prime \prime}=\rho^2 u$. The even solution is $u=\cosh \rho x$, and applying the boundary condition gives $\rho \tanh \rho t = -\alpha$. Hence $\alpha < 0$, and multiplying by $t$ gives $h_1(\rho t) = -\alpha t$. Inverting, $\rho = h_1^{-1}(-\alpha t)/t$ when $\alpha<0$. 

The odd solution is $u=\sinh \rho x$. Applying the boundary condition yields $\rho \coth \rho t = -\alpha$, or $h_2(\rho t) = -\alpha t$. Hence $- \alpha t > 1$, or $\alpha < -1/t$. Inverting yields $\rho = h_2^{-1}(-\alpha t)/t$. This $\rho$-value is smaller than the one found in the even case, since $h_2^{-1} < h_1^{-1}$ by \eqref{h1h2}, and hence the eigenvalue $\lambda=-\rho^2$ is larger than in the even case. 

There are no other negative eigenvalues. Combining these facts establishes the formula for $\lambda_1$ in \autoref{firsteigen} when $\alpha<0$, and the formula for $\lambda_2$ in \autoref{secondeigen} when $\alpha<-1/t$. 

(ii) Now suppose $\lambda=0$, so that the eigenfunction equation is $u^{\prime \prime}=0$. The even solution $u=1$ satisfies the boundary condition when $\alpha=0$, and the odd solution $u=x$ satisfies it when $\alpha=-1/t$. This yields the zero eigenvalues in the lemmas. 

(iii) Lastly, suppose $\lambda>0$, and write $\lambda= \rho^2$ where $\rho>0$. The eigenfunction equation $u^{\prime \prime}= -\rho^2 u$ has even solution $u=\cos \rho x$, for which the boundary condition says $\rho t \tan \rho t = \alpha t$. The roots of this condition arise from the branches of $x\tan x$, and so there are roots with $\rho t \in (\pi/2,3\pi/2), (3\pi/2,5\pi/2),\ldots$; and when $\alpha>0$ we can further narrow these intervals to $\rho t \in (\pi,3\pi/2), (2\pi,5\pi/2), \ldots$; also, when $\alpha>0$ there is a smaller root with $\rho t \in (0,\pi/2)$ coming from the first branch of $\tan$, that is, from $g_1(\rho t)=\alpha t$. Thus the smallest ``even'' eigenvalue when $\alpha>0$ is the square of $\rho = g_1^{-1}(\alpha t)/t$. 

Consider now the odd solution $u=\sin \rho x$ of the eigenfunction equation. It must satisfy the boundary condition $-\rho t \cot \rho t = \alpha t$. The roots of the boundary condition come from the branches of $-x\cot x$, and so there are roots with $\rho t \in (\pi,2\pi), (2\pi,3\pi)$ and so on; and when $\alpha > -1/t$ (so that $\alpha t > -1$ lies in the range of $g_2$), there is also a smaller root with $\rho t \in (0,\pi)$, coming from the first branch of $\cot$, that is, from $g_2(\rho t) = \alpha t$. More precisely, if $-1/t<\alpha \leq 0$ then $\rho t \in (0,\pi/2]$ and if $\alpha>0$ then $\rho t \in (\pi/2,\pi)$. Either way,  the smallest ``odd'' eigenvalue when $\alpha>-1/t$ is the square of $\rho = g_2^{-1}(\alpha t)/t$. 

Suppose $\alpha>0$. All eigenvalues are then positive, and the preceding paragraphs show the smallest even eigenvalue has $\rho t<\pi/2$, while the smallest odd eigenvalue has $\rho t>\pi/2$. Thus the even eigenvalue is the first one, which gives the formula for $\lambda_1$ in \autoref{firsteigen}. Also, the smallest odd eigenvalue has $\rho t < \pi$, while the second-smallest even eigenvalue has $\rho t > \pi$. Thus the odd eigenvalue is the smaller one, giving the formula for $\lambda_2$ in \autoref{secondeigen} when $\alpha>0$.

Suppose finally that $-1/t < \alpha \leq 0$. As found in parts (i) and (ii) of the proof, the first eigenvalue is even and $\leq 0$, while all other eigenvalues are positive. The work above shows that the smallest odd eigenvalue has $\rho t \leq \pi/2$, while the second-smallest even eigenvalue has $\rho t > \pi/2$. Again the odd eigenvalue is the smaller one, giving the formula for $\lambda_2$ in \autoref{secondeigen} when $-1/t < \alpha \leq 0$.

Finally, the first and second eigenvalues are strictly increasing as functions of $\alpha$ because $g_1, h_1, g_2, h_2$ and their inverses are all strictly increasing. The liming values as $\alpha \to \infty$ follow from evaluating $g_1^{-1}(\infty)=\pi/2$ and $g_2^{-1}(\infty)=\pi$. To derive the limiting behavior $-\alpha^2 \left(1+ O(e^{2\alpha t}) \right)$ as $\alpha \to -\infty$, simply substitute $y=-\alpha t$ into the asymptotic formulas in \autoref{inverseasymptotics}. 
\end{proof}
To determine qualitatively how the first two eigenvalues of the interval depend on its length, we split the next three propositions into the cases of $\alpha$ being positive, zero, or negative. \autoref{firsttwoint} illustrates the negative and positive cases.
\begin{figure}
\begin{center}
\includegraphics[scale=0.5]{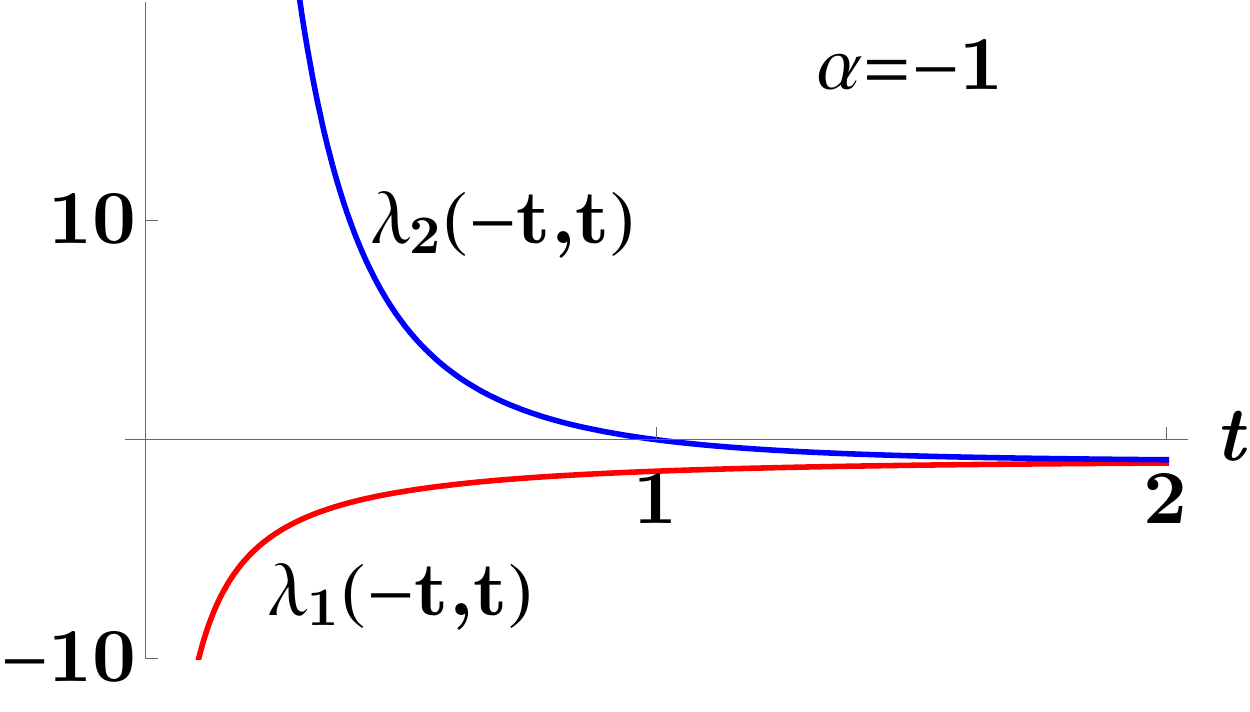}
\hspace{1cm}
\includegraphics[scale=0.5]{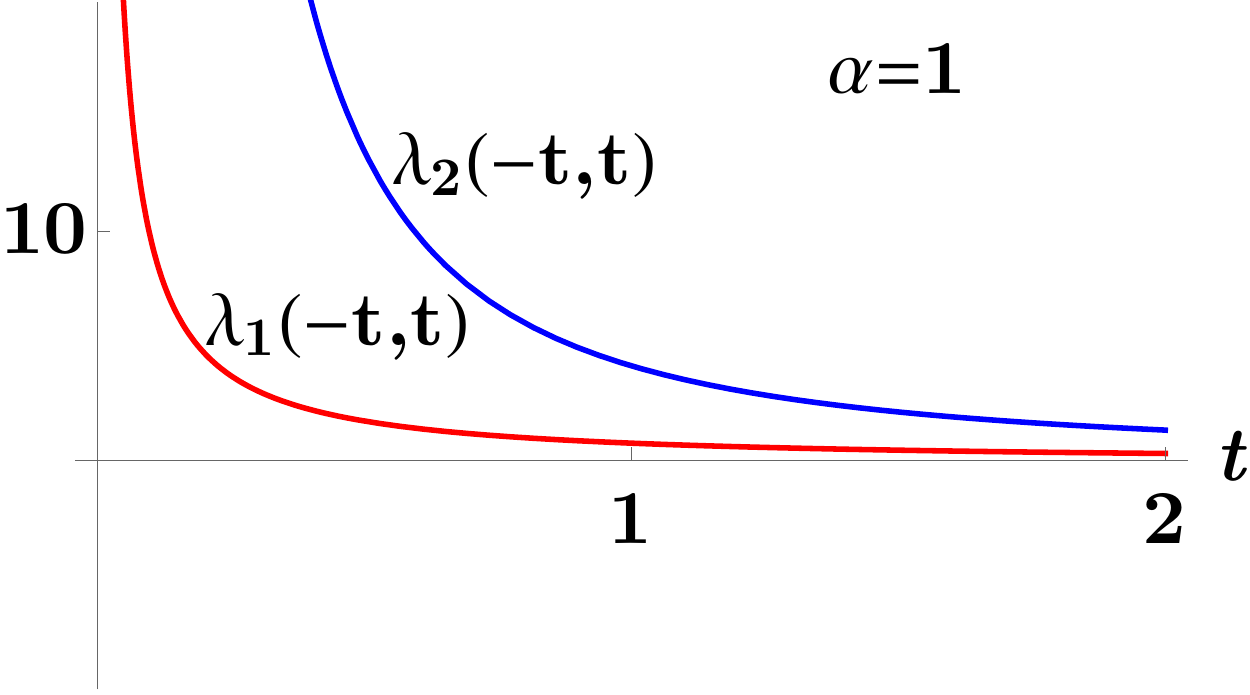}
\end{center}
\caption{\label{firsttwoint} Left: the first two eigenvalues $\lambda_1$ and $\lambda_2$ for the interval $\mathcal{I}(t)=(-t,t)$, when the half-length $t>0$ is variable and the Robin parameter $\alpha=-1$ is fixed. The horizontal asymptote has height $-\alpha^2=-1$. Right: the eigenvalues $\lambda_1$ and $\lambda_2$ as functions of $t>0$ when $\alpha=1$.}
\end{figure}

\begin{proposition}\label{1dimpos}
When $\alpha>0$, the first two eigenvalues, $\lambda_1(\mathcal{I}(t);\alpha)$ and $\lambda_2(\mathcal{I}(t);\alpha)$, are strictly decreasing as functions of $t>0$, and so is the spectral gap $\lambda_2(\mathcal{I}(t);\alpha)-\lambda_1(\mathcal{I}(t);\alpha)$. As $t$ increases from $0$ to $\infty$, all three functions decrease from $\infty$ to $0$.
\end{proposition}
See the right side of \autoref{firsttwoint}. In fact, every eigenvalue $\lambda_k(\mathcal{I}(t);\alpha)$ for $k \geq 1$ is decreasing as a function of $t$, when $\alpha>0$, as one can see by rescaling the integrals in the Rayleigh quotient to integrate over the fixed interval $(-1,1)$ instead of over $\mathcal{I}(t)$. The interesting part of the lemma is that the spectral gap also decreases with $t$. 
\begin{proof}
The function $\go$ is positive and strictly decreasing, by \autoref{monotonicity}, and so $t \mapsto \lambda_1(\mathcal{I}(t);\alpha)$ is positive and strictly decreasing, by the formula in \autoref{firsteigen}. It is easy to check that $\lim_{y \to 0+} \go(y)=\infty$ and $\lim_{y \to \infty} \go(y)=0$, and so $\lambda_1(\mathcal{I}(t);\alpha)$ tends to $0$ as $t \to \infty$ and tends to $\infty$ as $t \to 0$. (The blow-up as $t \to 0$ can be determined quite precisely, since $g_1(x) = x \tan x \simeq x^2$ as $x \to 0$ and so $g_1^{-1}(y)^2 \simeq y$, so that $\lambda_1(\mathcal{I}(t);\alpha) \simeq \alpha t /t^2=\alpha/t$ as $t \to 0$.)

The function $\gt(y)$ is positive and strictly decreasing for $y>0$, by \autoref{monotonicity}. Hence by \autoref{secondeigen}, $t \mapsto \lambda_2(\mathcal{I}(t);\alpha)$ is positive and strictly decreasing. Again it is straightforward to see $\lim_{y \to 0+} \gt(y)=\infty$ and $\lim_{y \to \infty} \gt(y)=0$, and so $\lambda_2(\mathcal{I}(t);\alpha)$ tends to $\infty$ as $t \to 0$ and tends to $0$ as $t \to \infty$.  

Next, decompose the spectral gap as 
\[
(\lambda_2-\lambda_1)(\mathcal{I}(t);\alpha) 
= \big( \sqrt{\lambda_2(\mathcal{I}(t);\alpha)}+\sqrt{\lambda_1(\mathcal{I}(t);\alpha)} \big) \big( \sqrt{\lambda_2(\mathcal{I}(t);\alpha)}-\sqrt{\lambda_1(\mathcal{I}(t);\alpha)} \big) .
\]
The first factor on the right side is strictly decreasing from $\infty$ to $0$ as a function of $t$, because $\lambda_1$ and $\lambda_2$ have that property. Meanwhile, the second factor equals $\alpha \gt(\alpha t) - \alpha \go(\alpha t)$, whose $t$-derivative is $\alpha^2 \big( \gt^\prime(\alpha t) - \go^\prime(\alpha t) \big)$. This derivative is negative by \eqref{Hderiv} in \autoref{derivcomparison}, and so the second factor decreases strictly as $t$ increases. \autoref{1dimpos} now follows. 
\end{proof}

The result is easy in the Neumann case, where $\alpha=0$:
\begin{proposition}\label{1dimzero}
When $\alpha=0$, the first eigenvalue $\lambda_1(\mathcal{I}(t);0)=0$ is constant and the second eigenvalue $\lambda_2(\mathcal{I}(t);0)=(\pi/2t)^2$ decreases strictly from $\infty$ to $0$ as $t$ increases from $0$ to $\infty$. 
\end{proposition}
Note the Neumann spectral gap equals the second eigenvalue, because the first eigenvalue is zero. 

Next we treat negative $\alpha$. Again see \autoref{firsttwoint}.
\begin{proposition}\label{1dimneg}
Fix $\alpha<0$. Then $\lambda_1(\mathcal{I}(t);\alpha)$ is strictly increasing and $\lambda_2(\mathcal{I}(t);\alpha)$ is strictly decreasing, as a function of $t>0$, and hence the spectral gap $\lambda_2(\mathcal{I}(t);\alpha) - \lambda_1(\mathcal{I}(t);\alpha)$ is strictly decreasing. The limiting values are:
\begin{align*}
 \lim_{t \to 0} \lambda_1(\mathcal{I}(t);\alpha) = -\infty , & \qquad  \lim_{t \to \infty} \lambda_1(\mathcal{I}(t);\alpha) = -\alpha^2 , \\
  \lim_{t \to 0} \lambda_2(\mathcal{I}(t);\alpha) = \infty , & \qquad  \lim_{t \to \infty} \lambda_2(\mathcal{I}(t);\alpha) = -\alpha^2 , \\
    \lim_{t \to 0} (\lambda_2-\lambda_1)(\mathcal{I}(t);\alpha) = \infty , & \qquad  \lim_{t \to \infty} (\lambda_2-\lambda_1)(\mathcal{I}(t);\alpha) = 0 ,
\end{align*}
and the horizontal intercept for $\lambda_2$ occurs at $t=1/|\alpha|$ since $\lambda_2(\mathcal{I}(1/|\alpha|);\alpha) = 0$.
 \end{proposition}
The observation that $t \mapsto \lambda_1(\mathcal{I}(t);\alpha)$ is strictly increasing, when $\alpha < 0$, was made already by Antunes \emph{et al.}\ \cite[Proposition 2]{AFK17}, and they found the limiting value $-\alpha^2$ as $t \to \infty$, in \cite[Proposition 3]{AFK17}.
\begin{proof}
The function $\ho$ is positive and strictly decreasing, by \autoref{monotonicity}, and so $t \mapsto \lambda_1(\mathcal{I}(t);\alpha)$ is negative and strictly increasing, by the formula in \autoref{firsteigen}. Since $\ho(\infty)=1$ and $\ho(0+)=\infty$, the limiting values of $\lambda_1$ as $t \to \infty$ and $t \to 0$ are as stated in the lemma. (The blow-up as $t \to 0$ can be established precisely, since $h_1(x) = x \tanh x \simeq x^2$ as $x \to 0$ and so $h_1^{-1}(y)^2 \simeq y$, so that $\lambda_1(\mathcal{I}(t);\alpha) \simeq -(-\alpha t) /t^2=\alpha/t$ as $t \to 0$. This blow-up rate was noted by Antunes \emph{et al.}\ \cite[Proposition 3]{AFK17}.)

The second eigenvalue requires more careful analysis. The function $\gt(y)$ is negative and strictly decreasing for $-1<y<0$, by \autoref{monotonicity}, and so $\gt(y)^2$ is positive and strictly increasing. Hence $t \mapsto \lambda_2(\mathcal{I}(t);\alpha)$ is positive and strictly decreasing when $0 < t < -1/\alpha$ by \autoref{secondeigen} (remembering here that $-\alpha>0$). Further, $\gt(0-)= -\infty$ and so $\lambda_2$ tends to $\infty$ as $t \to 0+$. Also, $\lim_{y \searrow -1} G_2(y)=0$ and so the eigenvalue approaches $0$ as $t$ approaches $-1/\alpha$ from below. 

When $t = -1/\alpha$ the second eigenvalue is $0$. 

Now suppose $t > -1/\alpha$. \autoref{monotonicity} says $\htwo$ is positive and strictly increasing, and so $t \mapsto \lambda_2(\mathcal{I}(t);\alpha)$ is negative and strictly decreasing by \autoref{secondeigen}. Note the eigenvalue approaches $0$ as $t$ approaches $-1/\alpha$ from above, since $\htwo(1)=0$. Further, $\htwo(\infty)=1$ and so $\lambda_2$ tends to $-\alpha^2$ as $t \to \infty$.
\end{proof}

\section{\bf The first and second Robin eigenvalues of a rectangular box}
\label{identifying}

Now that the interval is understood, we can identify the first and second Robin eigenvalues of the rectangular box
\[
\mathcal{B}(w) = \mathcal{I}(w_1) \times \dots \times \mathcal{I}(w_n)
\]
where $w=(w_1,\dots,w_n) \in \Rn, n \geq 1$, with $w_j > 0$ for each $j$. The width of the box in the $j$th direction is $2w_j$. Later in the section we show the spectral gap of the box is the same as the gap of its longest edge, that is, the largest width or longest interval. 
\begin{lemma}[First eigenvalue] \label{firsteigenbox}
\begin{align*}
\lambda_1(\mathcal{B}(w);\alpha)  
& = \lambda_1\big( \mathcal{I}(w_1) ; \alpha \big) + \lambda_1\big( \mathcal{I}(w_2) ; \alpha \big) + \dots + \lambda_1\big( \mathcal{I}(w_n) ; \alpha \big) \\
& = 
\begin{cases}
\alpha^2 \left| \big( \go(\alpha {w_1}), \dots , \go(\alpha {w_n}) \big) \right|^2 & \text{if $\alpha>0$,} \\
0 & \text{if $\alpha=0$,} \\
- \alpha^2 \left| \big( \ho(-\alpha {w_1}), \dots , \ho(-\alpha {w_n}) \big) \right|^2 & \text{if $\alpha<0$.}
\end{cases}
\end{align*}
This first eigenvalue is a strictly increasing function of $\alpha \in \R$. 
\end{lemma}
\begin{proof}
By separation of variables, the first eigenvalue for the box arises from summing the first eigenvalues of each of the intervals. (The first eigenfunction for the box is the product of the first eigenfunctions of the intervals.) Hence the lemma follows directly from \autoref{firsteigen}. 
\end{proof}
The first eigenvalue tends to infinity in magnitude when any one of the edge lengths tends to zero:
\begin{equation}\label{firsttoinfinity}
\lim_{w_n \to 0} \lambda_1(\mathcal{B}(w);\alpha) = 
\begin{cases}
\ \ \infty & \text{if $\alpha>0$,} \\
-\infty & \text{if $\alpha<0$,}
\end{cases}
\end{equation}
by \autoref{1dimpos} and \autoref{1dimneg}, where the other edges $w_1,\dots,w_{n-1}$ are arbitrary and may vary as $w_n \to 0$. For more precise inequalities on the first eigenvalue see Freitas and Kennedy \cite[Appendix A.1]{FK18}. 

The second eigenvalue of the box depends on knowing which edge is longest. 
\begin{lemma}[Second eigenvalue] \label{secondeigenbox}
If the longest edge of the box is the first one, so that $w_1 \geq w_j$ for all $j$, then 
\begin{equation} \label{eigentwodecomp}
\lambda_2\big( \mathcal{B}(w) ; \alpha \big) = \lambda_2\big( \mathcal{I}(w_1) ; \alpha \big) + \lambda_1\big( \mathcal{I}(w_2) ; \alpha \big) + \dots + \lambda_1\big( \mathcal{I}(w_n) ; \alpha \big) 
\end{equation}
for all $\alpha \in \R$. This second eigenvalue is a strictly increasing function of $\alpha \in \R$. 
\end{lemma}
It is no loss of generality to suppose the first edge of the box is the longest, since we may always rotate the box. Formula \eqref{eigentwodecomp} can be made more explicit by using the interval results from \autoref{identifyinginterval}. 
\begin{proof}
By separation of variables, the second eigenvalue for the box arises from summing the second eigenvalue on one of the intervals, say the $k$th interval, and the first eigenvalues of the remaining $n-1$ intervals. We will show $w_k \geq w_j$ for all $j$, so that the longest interval is the one on which the second eigenvalue must be taken. 

Since $\lambda_2\big( \mathcal{B}(w) ; \alpha \big)$ is the smallest eigenvalue having the specified form, the eigenvalue would increase if we used the second eigenvalue for $w_j$ instead of for $w_k$. Thus
\[
\lambda_2\big( \mathcal{I}(w_k) ; \alpha \big) + \lambda_1\big( \mathcal{I}(w_j) ; \alpha \big) 
\leq 
\lambda_1\big( \mathcal{I}(w_k) ; \alpha \big) + \lambda_2\big( \mathcal{I}(w_j) ; \alpha \big) .
\]
That is, the spectral gap of the interval increases from $w_k$ to $w_j$: 
\[
(\lambda_2-\lambda_1)\big( \mathcal{I}(w_k) ; \alpha \big) \leq
(\lambda_2-\lambda_1)\big( \mathcal{I}(w_j) ; \alpha \big) .
\]
Since the spectral gap is strictly decreasing as a function of the length of the interval, by \autoref{1dimpos}, \autoref{1dimzero} and \autoref{1dimneg}, we deduce $w_k \geq w_j$.
\end{proof}

\begin{corollary}[Spectral gap of a box equals the gap of its longest edge] \label{boxgap}
If $w_1 \geq w_j$ for all $j$ then 
\[
(\lambda_2-\lambda_1)\big( \mathcal{B}(w) ; \alpha \big) = (\lambda_2-\lambda_1)\big( \mathcal{I}(w_1) ; \alpha \big) , \qquad \alpha \in \R .
\]
\end{corollary}
This corollary follows by subtraction of \autoref{firsteigenbox} and \autoref{secondeigenbox}. 

\begin{example}[Second eigenvalue of the square] \label{squareexample}
The square $\mathcal{S}$ with edge length $2$ has vanishing second eigenvalue for   
\[
\alpha_0 \simeq -0.68825 , 
\]
meaning $\lambda_2(\mathcal{S};\alpha_0) = 0$. Hence $\lambda_1(\mathcal{S};\alpha) < 0 < \lambda_2(\mathcal{S};\alpha)$ whenever $\alpha \in (\alpha_0,0)$.
\begin{proof}
We need only consider $\alpha<0$, since the second eigenvalue is positive when $\alpha \geq 0$.  From \autoref{secondeigenbox}, \autoref{firsteigen} and \autoref{secondeigen} we find
\begin{align*}
\lambda_2(\mathcal{S};\alpha)
& = \lambda_2\big( \mathcal{I}(1) ; \alpha \big) + \lambda_1\big( \mathcal{I}(1) ; \alpha \big) \\
& = g_2^{-1}(\alpha)^2 - h_1^{-1}(-\alpha)^2 .
\end{align*}
We assume here that $\alpha> -1$, since otherwise the lemmas show the second eigenvalue of the square is negative, whereas we want it to vanish. 

Thus the second eigenvalue vanishes when the number $\alpha \in (-1,0)$ satisfies $g_2^{-1}(\alpha) = h_1^{-1}(-\alpha)$. Writing $x=g_2^{-1}(\alpha) \in (0,\pi/2)$, the condition becomes $h_1(x)=-g_2(x)$, which reduces to $\tanh x = \cot x$. Solving numerically gives $x \simeq 0.93755$, and so $\alpha= g_2(x) = -x \cot x \simeq -0.68825$. 

Since the first and second Robin eigenvalues of the square are strictly increasing as functions of $\alpha$ by \autoref{firsteigenbox} and \autoref{secondeigenbox}, we conclude the second eigenvalue is positive when $\alpha > \alpha_0$, and of course the first eigenvalue is negative when $\alpha<0$. 
\end{proof}
\end{example}

\section{\bf Proofs of main theorems}
\label{mainproofs}
\begin{proof}[\bf Proof of \autoref{gapmonot}]
Without loss of generality we may assume $w_1$ is the largest of the $w_j$. We will show that the spectral gap is strictly increasing for $\alpha$ in each of the three intervals $(-\infty,-1/{w_1}), (-1/{w_1},0)$ and $(0,\infty)$. 

The spectral gap of the box equals the spectral gap of its longest side, with
\[
(\lambda_2-\lambda_1)(\mathcal{B};\alpha) = (\lambda_2-\lambda_1)(\mathcal{I}(w_1);\alpha)
\]
by \autoref{boxgap}. Hence when $\alpha>0$, 
\[
(\lambda_2-\lambda_1)(\mathcal{B};\alpha) = \frac{g_2^{-1}(\alpha {w_1})^2 - g_1^{-1}(\alpha {w_1})^2}{w_1^2} 
\]
by using the formulas for the first two eigenvalues of the interval from \autoref{firsteigen} and \autoref{secondeigen}. 
Thus the spectral gap is strictly increasing with respect to $\alpha>0$, by \eqref{inversederiv} in \autoref{derivcomparison}. The limit as $\alpha \to \infty$ equals $\big(\pi^2 - (\pi/2)^2\big)/w_1^2$,
which is the gap between the first two Dirichlet eigenvalues of the box. 

When $-1/{w_1}<\alpha<0$, the gap is
\[
(\lambda_2-\lambda_1)(\mathcal{B};\alpha) = \frac{g_2^{-1}(\alpha {w_1})^2 + h_1^{-1}(-\alpha {w_1})^2}{w_1^2} ,
\]
which is strictly increasing with respect to $\alpha$ by \autoref{derivcomparisonneg}(i). 

When $\alpha<-1/{w_1}$, the gap formula is that 
\[
(\lambda_2-\lambda_1)(\mathcal{B};\alpha) = \frac{-h_2^{-1}(-\alpha {w_1})^2 + h_1^{-1}(-\alpha {w_1})^2}{w_1^2} ,
\]
which is strictly increasing with respect to $\alpha$ by \autoref{derivcomparisonneg}(ii). The gap tends to $0$ as $\alpha \to -\infty$, by the asymptotic formulas for the interval stated in \autoref{firsteigen} and \autoref{secondeigen}. 
\end{proof}
\begin{proof}[\bf Proof of \autoref{ratiomonot}]
The first eigenvalue equals $0$ at $\alpha=0$, and is positive when $\alpha>0$ and negative when $\alpha<0$. 
Further, it is concave as a function of $\alpha$, as we observed in \autoref{results} using the characterization of $\lambda_1(\Omega;\alpha)$ as the minimum of the Rayleigh quotient (which  depends linearly on $\alpha$). Hence the difference quotient 
\[
\frac{\lambda_1(\Omega;\alpha)}{\alpha} = \frac{\lambda_1(\Omega;\alpha)-\lambda_1(\Omega;0)}{\alpha-0} 
\]
is positive for all $\alpha \neq 0$, and is decreasing as a function of $\alpha$, by concavity.

The theorem now follows, since 
\[
\alpha \frac{\lambda_2(\Omega;\alpha)}{\lambda_1(\Omega;\alpha)} = \frac{\lambda_2(\Omega;\alpha)}{\lambda_1(\Omega;\alpha)/\alpha} 
\]
where on the right side both the numerator and the denominator are positive, and the numerator is increasing and the denominator is decreasing as a function of $\alpha$, so that the ratio is increasing. 
\end{proof}

\begin{proof}[\bf Proof of \autoref{secondeigconcave}]
\autoref{firsteigenbox} says that the first eigenvalue of the box is found by summing the first eigenvalues of each edge, and similarly for the second eigenvalue of the box in \autoref{secondeigenbox} except in that case one uses the second eigenvalue of the longest edge. Thus it suffices to establish the $1$-dimensional case of the theorem, namely, to show strict concavity with respect to $\alpha$ of the first and second eigenvalues of a fixed interval $\mathcal{I}(t)$. 

If $\alpha > 0$ then $\lambda_1 \big( \mathcal{I}(t); \alpha \big)
= g_1^{-1}(\alpha t)^2/t^2$ by \autoref{firsteigen}, and so \autoref{concaveinverse} gives strict concavity with respect to $\alpha$. If $\alpha < 0$ then $\lambda_1 \big( \mathcal{I}(t); \alpha \big) = - h_1^{-1}(-\alpha t)^2/t^2$ and so again \autoref{concaveinverse}  yields strict concavity. To ensure concavity of the first eigenvalue around the ``join'' at $\alpha=0$, we note the slopes match up from the left and the right there: $g_1(x) \simeq x^2$ and $h_1(x) \simeq x^2$ for $x \simeq 0$, and so $\lambda_1 \big( \mathcal{I}(t); \alpha \big) \simeq \alpha/t$ when $\alpha \simeq 0$. 

If $\alpha > -1/t$ then $\lambda_2 \big( \mathcal{I}(t); \alpha \big) = g_2^{-1}(\alpha t)^2/t^2$ by \autoref{secondeigen} and so \autoref{concaveinverse} proves strict concavity with respect to $\alpha$. If $\alpha < -1/t$ then $\lambda_2 \big( \mathcal{I}(t); \alpha \big) = - h_2^{-1}(-\alpha t)^2/t^2$ and again \autoref{concaveinverse} proves strict concavity. 

For concavity of the second eigenvalue around the join at $\alpha=-1/t$, we will show the slopes from the left and right agree. For the right, we note that $g_2(x) = -x \cot x \simeq -1+x^2/3$ when $x \simeq 0$ and so $g_2^{-1}(y) \simeq \sqrt{3(1+y)}$, hence $\lambda_2 \big( \mathcal{I}(t); \alpha \big) \simeq (3/t)(\alpha+1/t)$ when $\alpha \simeq -1/t$. For the left, $h_2(x) = x \coth x \simeq 1+x^2/3$ when $x \simeq 0$ and so $h_2^{-1}(y) \simeq \sqrt{3(y-1)}$, and hence once again $\lambda_2 \big( \mathcal{I}(t); \alpha \big) \simeq (3/t)(\alpha+1/t)$ when $\alpha \simeq -1/t$. Thus the slopes of the second eigenvalue curve from the left and right are the same at $\alpha=-1/t$, namely $3/t$. Therefore, by our work above, strict concavity holds on a neighborhood of that point, completing the proof. 
\end{proof}

In order to prove the next theorem, we need an elementary convexity result for the norm of a separated vector field. 
\begin{lemma}\label{convexfield}
If $f_1,\dots,f_n$ are nonnegative, strictly convex functions on $\R$ then 
\[
\left| \big( f_1(z_1),\dots,f_n(z_n) \big) \right|
\]
is strictly convex as a function of $z=(z_1,\dots,z_n) \in \Rn$.
\end{lemma}
\begin{proof}
If $w=(w_1,\dots,w_n)$ and $z=(z_1,\dots,z_n)$ are given and $0 < \e < 1$, then by the triangle inequality, 
\begin{align*}
& (1-\e)\left| \big( f_1(w_1),\dots,f_n(w_n) \big) \right| + \e \left| \big( f_1(z_1),\dots,f_n(z_n) \big) \right| \\
 & \geq \left| \big( (1-\e)f_1(w_1) + \e f_1(z_1),\dots,(1-\e)f_n(w_n) + \e f_n(z_n) \big) \right| \\
 & \geq \left| \big( f_1((1-\e)w_1+\e z_1),\dots,f_n((1-\e)w_n+\e z_n) \big) \right| 
\end{align*}
by convexity of $f_1,\dots,f_n$ and the fact that all components of the vectors are nonnegative. Further, if equality holds then $w_1=z_1,\dots,w_n=z_n$ by  strict convexity of $f_1,\dots,f_n$. Thus strict convexity holds in the lemma. 
\end{proof}

\begin{proof}[\bf Proof of \autoref{lambda1higherdim}]
We start with convexity results for the first eigenvalue. Given a vector $z = (z_1,\dots,z_n) \in \Rn$, write
\[
e^z = (e^{z_1},\dots,e^{z_n}) .
\]
We will prove:
\begin{align}
\text{if $\alpha>0$ then $\sqrt{\lambda_1(\mathcal{B}(e^z);\alpha)}$ is a strictly convex function of $z \in \Rn$,} \label{lambda1sqrti} \\
\text{if $\alpha<0$ then $\sqrt{-\lambda_1(\mathcal{B}(e^z);\alpha)}$ is a strictly convex function of $z \in \Rn$.} \label{lambda1sqrtii}
\end{align}
First, \autoref{firsteigenbox} gives when $\alpha>0$ that
\[
\sqrt{\lambda_1(\mathcal{B}(e^z);\alpha)} 
= \alpha \left| \big( \go(\alpha e^{z_1}), \dots , \go(\alpha e^{z_n}) \big) \right| .
\]
Each individual component $\go(\alpha e^{z_j})$ is strictly convex as a function of $z_j$ by \autoref{convexity}, and so \autoref{convexfield} implies conclusion \eqref{lambda1sqrti}. Simlarly, when $\alpha<0$ we have
\[
\sqrt{-\lambda_1(\mathcal{B}(e^z);\alpha)} 
= |\alpha| \left| \big( \ho(|\alpha| e^{z_1}), \dots , \ho(|\alpha| e^{z_n}) \big) \right| .
\]
The components $\ho(|\alpha| e^{z_j})$ are strictly convex as functions of $z_j$, by \autoref{convexity}, and so conclusion \eqref{lambda1sqrtii} follows from \autoref{convexfield}. Now we can prove the theorem. 

(i) Suppose $\alpha>0$. Consider rectangular boxes $\mathcal{B}(e^z)$ of given volume $V$, which means $2e^{z_1} \cdots 2e^{z_n}=V$, or $z_1 + \dots + z_n = \log(2^{-n}V)$. This set of vectors $z$ forms a hyperplane in $\Rn$ perpendicular to the direction $(1,\dots,1)$, and the function $f(z)=\sqrt{\lambda_1(\mathcal{B}(e^z);\alpha)}$ is strictly convex on that hyperplane by \eqref{lambda1sqrti}. 

We want to show $f$ achieves its strict global minimum at the cube. That is, we want $f$ to have a strict global minimum at the point $z=(t,\dots,t)$ where the hyperplane intersects the line through the origin in direction $(1,\dots,1)$. Due to the strict convexity, it suffices to show that the gradient of $f$ restricted to the hyperplane vanishes at this $z$, which means we want $(\nabla f) (t,\dots,t)$ to be parallel to $(1,\dots,1)$. That the gradient vector has this property follows from the invariance of $f(z_1,\dots,z_n)$ under permutation of the variables. 

\emph{Note.} Convexity of $\lambda_1(\mathcal{B}(e^z);\alpha)$ was proved by Keady and Wiwatanapataphee \cite[Corollary 2]{KW18} when $\alpha>0$. That convexity is weaker than \eqref{lambda1sqrti}, where the square root is imposed on the eigenvalue, but it was strong enough for them to prove part (i) of \autoref{lambda1higherdim}.

\smallskip
(ii) Suppose $\alpha<0$. Argue as in part (i), except this time using \eqref{lambda1sqrtii} instead of \eqref{lambda1sqrti} and letting $f(z)=\sqrt{-\lambda_1(\mathcal{B}(e^z);\alpha)}$. 
\end{proof}
\begin{proof}[\bf Proof of \autoref{lambda1L}]
By scale invariance of the expression $\lambda_1(\mathcal{R};\alpha/L)A$ we may assume the rectangle has perimeter $L=2$. That is, we need only consider the family of rectangles 
\[
\mathcal{R}(p) = (0,p) \times (0,1-p) ,
\]
where $0 < p < 1$. Clearly these rectangles have perimeter $2$ and area $p(1-p)$.

(i) First suppose $\alpha>0$. We claim $\lambda_1 \big( \mathcal{R}(p);\alpha/2 \big) A \big( \mathcal{R}(p) \big)$ is strictly convex as a function of $p \in (0,1)$, and hence is strictly decreasing for $p \in (0,1/2]$ and strictly increasing for $p \in [1/2,1)$, with its minimum at $p=1/2$ (the square).

By \autoref{firsteigenbox} applied with $\alpha/2$ instead of $\alpha$, and with ${w_1}=p/2$ and ${w_2}=(1-p)/2$, we have 
\[
\lambda_1(\mathcal{R}(p);\alpha/2) A \big( \mathcal{R}(p) \big) = (\alpha/2)^2 \left( \go(\alpha p/4)^2 + \go(\alpha (1-p)/4)^2 \right) p(1-p) .
\]
The function $p \mapsto \go(\alpha p/4)^2 \, p(1-p)$ is strictly convex for $0<p<1$ by \autoref{convexityyy}(i), and replacing $p$ by $1-p$ shows that $p \mapsto \go(\alpha (1-p)/4)^2 \, p(1-p)$ is strictly convex also. Clearly $\lambda_1 \big( \mathcal{R}(p);\alpha/2 \big) A \big( \mathcal{R}(p) \big)$ is even with respect to $p=1/2$ since the rectangle $\mathcal{R}(1-p)$ is the same as $\mathcal{R}(p)$ except rotated by angle $\pi/2$. Thus by the strict convexity just proved, the function $p \mapsto \lambda_1 \big( \mathcal{R}(p);\alpha/2 \big) A \big( \mathcal{R}(p) \big)$ must be strictly decreasing for $p \in (0,1/2]$ and strictly increasing for $p \in [1/2,1)$. 

(ii) Next suppose $\alpha<0$. We claim $\lambda_1 \big( \mathcal{R}(p);\alpha/2 \big) A \big( \mathcal{R}(p) \big)$ is strictly decreasing for $p \in (0,1/2]$ and strictly increasing for $p \in [1/2,1)$, so that again the minimum occurs for the square, $p=1/2$. 

\autoref{firsteigenbox} with ${w_1}=p/2$ and ${w_2}=(1-p)/2$ gives that  
\begin{equation} \label{firstLeq}
-\lambda_1(\mathcal{R}(p);\alpha/2) A \big( \mathcal{R}(p) \big) = (\beta/2)^2 \left( \ho(\beta p/4)^2 + \ho(\beta (1-p)/4)^2 \right) p(1-p)
\end{equation}
where $\beta=-\alpha>0$. To prove the claim it suffices to show the existence of a number $p(\beta)$ with 
\[
0 \leq p(\beta) < \frac{1}{2} 
\]
such that the right side of \eqref{firstLeq} is strictly increasing on $\big( 0, p(\beta) \big)$, strictly concave on $\big( p(\beta),1-p(\beta) \big)$, and strictly decreasing on $\big( 1-p(\beta) , 1\big)$ --- because then the evenness of \eqref{firstLeq} under $p \mapsto 1-p$ guarantees that the right side of \eqref{firstLeq} is strictly increasing on $(0,1/2]$ and strictly decreasing on $[1/2,1)$. 

In fact, we need only show that the term $p \mapsto \ho(\beta p/4)^2 p(1-p)$ is strictly increasing on $\big( 0, p(\beta) )$, strictly concave on $\big( p(\beta),1-p(\beta) \big)$, and strictly decreasing on $\big( 1-p(\beta),1 \big)$, because then the same holds true when we replace $p$ by $1-p$, and adding two functions with these properties yields another function with these properties. 

\autoref{concavityH1} establishes the desired properties with $p(\beta)=y_1(\beta/4)$, and in fact establishes a little more, namely that $\ho(\beta p/4)^2 p(1-p)$ is strictly concave and strictly decreasing on the whole interval $\big( p(\beta),1 \big)$. Thus the theorem is proved. 
\end{proof}
\begin{proof}[\bf Proof of \autoref{linearbound}]
We extend the $2$-dimensional proof given by Freitas and Laugesen \cite[Theorem A]{FL18b}. Substituting the constant trial function $u(x) \equiv 1$ into the Rayleigh quotient gives the upper bound
\[
   \lambda_1(\Omega;\alpha V^{1-2/n}/S) V^{2/n} \leq \frac{0 + (\alpha V^{1-2/n}/S)\int_{\partial \Omega} 1^2 \, dS}{\int_\Omega 1^2 \, dx} \, V^{2/n} = \alpha.
\]
We show this inequality must be strict. If equality held, then the constant trial function $u \equiv 1$ would be a first eigenfunction, and taking the Laplacian of it would imply $\lambda_1(\Omega;\alpha V^{1-2/n}/S) = 0$, and hence $\alpha=0$, contradicting the hypothesis in the theorem. Hence equality cannot hold and the inequality is strict. 

To show equality is attained asymptotically for rectangular boxes that degenerate, consider a box $\mathcal{B}(w)$ and assume the volume is fixed, say $V=1$ for convenience. Suppose the box degenerates, which means the surface area tends to infinity. The surface area is 
\[
S = 2 \sum_{k=1}^n \frac{2w_1 \dots 2w_n}{2w_k} = \sum_{k=1}^n \frac{1}{w_k} ,
\]
since $V=2w_1 \cdot \dots \cdot 2w_n=1$. 
For $\alpha>0$ we have
\begin{align*}
\lambda_1\big(\mathcal{B}(w);\alpha/S\big) 
& = \sum_{k=1}^n g_1^{-1}(w_k \alpha/S)^2/w_k^2 \qquad \text{by \autoref{firsteigenbox}} \\
& \geq \alpha - n\alpha^2/S^2 \qquad \text{since $g_1^{-1}(y)^2 \geq y-y^2$ by \autoref{inversebounds}} \\
& \to \alpha \qquad \qquad \qquad \text{as $S \to \infty$.}
\end{align*}
When $\alpha<0$ the proof is similar, except replacing $g_1^{-1}(w_k \alpha/S)^2$ with $-h_1^{-1}(w_k |\alpha|/S)^2$. 
\end{proof}
\begin{proof}[\bf Proof of \autoref{lambda2higherdim}]
The second eigenvalue of the box is 
\begin{equation} \label{eigentwo}
\lambda_2\big( \mathcal{B}(w) ; \alpha \big) = \lambda_2\big( \mathcal{I}(w_1) ; \alpha \big) + \lambda_1\big( \mathcal{I}(w_2) ; \alpha \big) + \dots + \lambda_1\big( \mathcal{I}(w_n) ; \alpha \big) 
\end{equation}
by \autoref{secondeigenbox}, where we take $w_1$ to be the largest of the $w_j$, that is, we assume the first edge of the box is its longest. 

When $\alpha=0$ (the Neumann case), the theorem is easy and well known: 
\[
\lambda_2\big( \mathcal{B}(w) ; \alpha \big) = \left( \frac{\pi}{2w_1} \right)^{\! \! 2} 
\]
and this expression is largest when the box is a cube having the same volume as the original box $\mathcal{B}(w)$, because in that case the longest side is as short as possible. 

Next suppose $\alpha<0$. We proceed in two steps. First we equalize the shorter edges of the box. Let 
\[
w_2^* = \dots = w_n^* =(w_2 \cdots w_n)^{1/(n-1)} 
\]
so that $w_j^* \leq w_1$, and define $\widehat{w}=(w_2,\dots,w_n) , \widehat{w}^*=(w_2^*,\dots,w_n^*) \in \R^{n-1}$. The $(n-1)$-dimensional boxes $\mathcal{B}(\widehat{w})$ and $\mathcal{B}(\widehat{w}^*)$ have the same volume, since $w_2 \cdots w_n = w_2^* \cdots w_n^*$. Formula \eqref{eigentwo} and maximality of the cube for the first eigenvalue when $\alpha<0$, from \autoref{lambda1higherdim}(ii), together show that 
\begin{align*}
\lambda_2\big( \mathcal{B}(w) ; \alpha \big) 
& = \lambda_2\big( \mathcal{I}(w_1) ; \alpha \big) + \lambda_1\big( \mathcal{B}(\widehat{w}) ; \alpha \big) \\
& \leq \lambda_2\big( \mathcal{I}(w_1) ; \alpha \big) + \lambda_1\big( \mathcal{B}(\widehat{w}^*) ; \alpha \big) ,
\end{align*}
with equality if and only if $w_2=\dots =w_n$. 

Next we equalize the first edge as well. Let  
\[
t = (w_1 w_2 \cdots w_n)^{1/n} = (w_1 w_2^* \cdots w_n^*)^{1/n} ,
\]
so that $w_j^* \leq t \leq w_1$ for each $j$. Then 
\begin{align*}
\lambda_2\big( \mathcal{I}(w_1) ; \alpha \big) + \lambda_1\big( \mathcal{B}(\widehat{w}^*) ; \alpha \big) 
& = \lambda_2\big( \mathcal{I}(w_1) ; \alpha \big) + \lambda_1\big( \mathcal{I}(w_2^*) ; \alpha \big) + \dots + \lambda_1\big( \mathcal{I}(w_n^*) ; \alpha \big) \\
& \leq \lambda_2\big( \mathcal{I}(t) ; \alpha \big) + \lambda_1\big( \mathcal{I}(t) ; \alpha \big) + \dots + \lambda_1\big( \mathcal{I}(t) ; \alpha \big)
\end{align*}
by the strict monotonicity properties of $\lambda_1$ and $\lambda_2$ with respect to the length of the interval, in \autoref{1dimneg}, when $\alpha < 0$. Equality holds if and only if $w_1=t$ and $w_j^*=t$. Putting together our inequalities, we conclude  
\[
\lambda_2\big( \mathcal{B}(w) ; \alpha \big) \leq \lambda_2\big( \mathcal{B}(t,\dots,t) ; \alpha \big) 
\]
with equality if and only if $w=(t,\dots,t)$. That is, $\lambda_2(\mathcal{B};\alpha)$ is maximal for the cube and only the cube, among rectangular boxes $\mathcal{B}$ of given volume. 
\end{proof}

\begin{proof}[\bf Proof of \autoref{sigma1higherdim}]
The Steklov eigenvalue problem for the Laplacian is
\[
\begin{split}
\Delta u & = 0 \ \ \quad \text{in $\Omega$,} \\
\frac{\partial u}{\partial\nu} & = \sigma u \quad \text{on $\partial \Omega$,} 
\end{split}
\]
where the eigenvalues are $0=\sigma_0 < \sigma_1 \leq \sigma_2 \leq \dots$. Clearly $\sigma$ belongs to the Steklov spectrum exactly when $0$ belongs to the Robin spectrum for parameter $\alpha=-\sigma$. In particular, $\alpha=-\sigma_1$ is the horizontal intercept value for the second Robin spectral curve, meaning $\lambda_2(\Omega;-\sigma_1)=0$. 

Let $\mathcal{C}$ be a cube having the same volume as the box $\mathcal{B}$. \autoref{lambda2higherdim} says $\lambda_2$ is smaller for $\mathcal{B}$ than for $\mathcal{C}$, at each $\alpha$, and since the eigenvalues are increasing with respect to $\alpha$, we conclude the horizontal intercept is larger (less negative) for $\mathcal{B}$ than for $\mathcal{C}$. In other words, $\sigma_1(\mathcal{B}) \leq \sigma_1(\mathcal{C})$. The inequality is strict due to the strictness in \autoref{lambda2higherdim}. 

A more detailed account of this proof goes as follows. From \autoref{secondeigenbox} and results in \autoref{identifyinginterval} we know $\lambda_2(\mathcal{B};\alpha)$ is continuous and strictly increasing as a function of $\alpha$, and tends to $-\infty$ as $\alpha \to -\infty$, and is positive at $\alpha=0$. Hence there is a unique horizontal intercept value $\alpha_{\mathcal{B}}<0$ at which $\lambda_2(\mathcal{B};\alpha_{\mathcal{B}})=0$. Note $\sigma_1(\mathcal{B})=-\alpha_{\mathcal{B}}$, since the fact that $\lambda_1(\mathcal{B};\alpha) < 0 <\lambda_2(\mathcal{B};\alpha)$ for all $\alpha \in (\alpha_{\mathcal{B}},0)$ implies that no $\alpha$-value in that interval corresponds to a Steklov eigenvalue for $\mathcal{B}$. Similarly there is a unique horizontal intercept value $\alpha_{\mathcal{C}}<0$ at which $\lambda_2(\mathcal{C};\alpha_{\mathcal{C}})=0$, and $\sigma_1(\mathcal{C})=-\alpha_{\mathcal{C}}$. 

Choosing $\alpha=\alpha_{\mathcal{B}}$ in \autoref{lambda2higherdim} gives that
\[
0 = \lambda_2(\mathcal{B};\alpha_{\mathcal{B}}) \leq \lambda_2(\mathcal{C};\alpha_{\mathcal{B}}) ,
\]
with strict inequality unless the box $\mathcal{B}$ is a cube. Because the eigenvalues are strictly increasing functions of $\alpha$, it follows that $\alpha_{\mathcal{C}} \leq \alpha_{\mathcal{B}}$ with strict inequality unless the box is a cube. That is, $\sigma_1(\mathcal{C}) \geq \sigma_1(\mathcal{B})$ with strict inequality unless the box is a cube. 
\end{proof}

\begin{proof}[\bf Proof of \autoref{lambda2L}] By scale invariance, it suffices to prove the theorem for the family of rectangles $\mathcal{R}(p) = (0,p) \times (0,1-p)$. These rectangles have perimeter $L=2$ and area $A=p(1-p)$, and so the quantity to be maximized is
\[
Q(p) = \lambda_2\big( \mathcal{R}(p);\alpha/2\big) p(1-p) .
\]
We may assume $p \in (0,1/2]$, so that the long side has length $1-p$ and the short side has length $p$. Then by  \autoref{secondeigenbox}, the second eigenvalue of the rectangle equals 
\begin{equation} \label{normalizedsecond}
\lambda_2\big( \mathcal{R}(p);\alpha/2\big) = \lambda_2\big( (0,1-p); \alpha/2\big) + \lambda_1\big( (0,p); \alpha/2\big) , \qquad p \in (0,1/2] .
\end{equation}

Step 1. We start by proving inequalities for the square ($p=1/2$), specifically that
\begin{align} 
\lambda_2\big( \mathcal{R}(1/2);\alpha/2\big)/4 & > \alpha, \qquad \alpha \in (\alpha_-,\alpha_+) , \label{biggeralpha1} \\
\lambda_2\big( \mathcal{R}(1/2);\alpha/2\big)/4 & < \alpha, \qquad \alpha \notin [\alpha_-,\alpha_+] . \label{biggeralpha2} 
\end{align}
Equality in \eqref{biggeralpha1} would mean
\begin{equation} \label{eq:concavesquare}
\lambda_2\big( \mathcal{I}(1/4) ; \alpha/2 \big) + \lambda_1\big( \mathcal{I}(1/4) ; \alpha/2 \big) = 4\alpha ,
\end{equation}
which when $\alpha>0$ reduces to 
\[
g_2^{-1}(\alpha/8)^2 + g_1^{-1}(\alpha/8)^2 = \alpha/4
\]
by applying the interval eigenvalue formulas in \autoref{firsteigen} and \autoref{secondeigen}. Thus equality holds at $\alpha_+ \simeq 33.2$ by definition \eqref{alphaplus}. When $\alpha<-8$, equality \eqref{eq:concavesquare} reduces to 
\[
-h_2^{-1}(-\alpha/8)^2 - h_1^{-1}(-\alpha/8)^2 = \alpha/4 , 
\]
and so equality holds at $\alpha_- \simeq -9.4$ by definition \eqref{alphaneg}. The strict inequalities \eqref{biggeralpha1} and \eqref{biggeralpha2} now follow from strict concavity of the second eigenvalue of the fixed rectangle $\mathcal{R}(1/2)$ as a function of $\alpha \in \R$ (\autoref{secondeigconcave}).  

\smallskip
Step 2. Next we establish convexity facts for the interval, on various ranges of $\alpha$-values. If $\alpha>0$ then 
\begin{align}
p \mapsto \lambda_1\big( (0,p); \alpha\big) p(1-p) \qquad \text{is strictly convex for $p \in (0,1)$,} \label{firstconvex} \\
p \mapsto \lambda_2\big( (0,p); \alpha\big) p(1-p) \qquad \text{is strictly convex for $p \in (0,1)$.} \label{secondconvex} 
\end{align}
Claim \eqref{firstconvex} holds by \autoref{convexityyy}(i), since $\lambda_1\big( (0,p); \alpha\big) = \alpha^2 \go( \alpha p/2)^2$ by \autoref{firsteigen} applied with $t=p/2$. Similarly claim \eqref{secondconvex} holds by \autoref{convexityyy}(ii), since $\lambda_2\big((0,p);\alpha \big) = \alpha^2 \gt(\alpha p/2)^2$ by \autoref{secondeigen}. 

If $-6 \leq \alpha < 0$ then 
\begin{equation} \label{specialalpha}
\text{$p \mapsto \lambda_1\big( (0,p); \alpha\big) p(1-p)$ is strictly increasing and strictly convex for $p \in (0,1)$,}
\end{equation}
by applying \autoref{concavityH1} with $c=|\alpha|/2 \leq 3$. 

If $\alpha < -6$ then 
\begin{align} 
\text{$p \mapsto \lambda_1\big( (0,p); \alpha\big) p(1-p)$ is strictly decreasing for\ } & \text{$p \in \big(0,y_1(|\alpha|/2)\big)$} \label{specialalpha6} \\
\text{and is strictly increasing and strictly convex for\ } & \text{$p \in \big(y_1(|\alpha|/2),1\big)$,} \label{specialalpha7}
\end{align}
by \autoref{concavityH1} applied with $c=|\alpha|/2 > 3$. The lemma showed $0<y_1(|\alpha|/2)<1/2$.

Now we claim when $\alpha < 0$ that the second eigenvalue satisfies:
\begin{align} 
\text{$p \mapsto \lambda_2\big( (0,p); \alpha\big) p(1-p)$ is strictly decreasing when $p \in \big(0,\min(1,2/|\alpha|) \big)$,} \label{specialalpha2} \\
\text{$\lambda_2\big( (0,p); \alpha\big) < 0$ when $p \in \big(\! \min(1,2/|\alpha|),1 \big)$.} \label{specialalpha3}
\end{align}
Indeed, if $p<2/|\alpha|$ then $\alpha>-2/p$ and so $\lambda_2\big( (0,p); \alpha\big) = \alpha^2 \gt(\alpha p/2)^2$ by \autoref{secondeigen}. Thus \eqref{specialalpha2} holds by \autoref{convexityyy}(iii) with $c=|\alpha|/2$. For \eqref{specialalpha3}, if $p>2/|\alpha|$ then $\alpha<-2/p$ and so  $\lambda_2\big( (0,p); \alpha\big) = -\alpha^2 \htwo(-\alpha p/2)^2 < 0$ by \autoref{secondeigen}. 

Further, if $\alpha < -2$ then the second eigenvalue satisfies that  
\begin{align}
\text{$p \mapsto \lambda_2\big( (0,p); \alpha\big) p(1-p)$ is strictly convex when\ } & \text{$p \in \big( 2/|\alpha|,1 \big)$} \label{specialalpha4} \\
\text{and strictly increasing when\ } & \text{$p \in \big( y_2(|\alpha|/2),1 \big)$.} \label{specialalpha5} 
\end{align}
Here, \autoref{concavityhtwo} with $c=|\alpha|/2>1$ ensures that $y_2(|\alpha|/2) \geq 2/|\alpha|$ and so the $p$ values in \eqref{specialalpha4} and \eqref{specialalpha5} satisfy $p>2/|\alpha|$. Hence $\lambda_2\big( (0,p); \alpha\big)= -\alpha^2 \htwo(|\alpha| p/2)^2$, and applying \autoref{concavityhtwo} yields \eqref{specialalpha4} and \eqref{specialalpha5}. That lemma also gives that 
\begin{equation} \label{y1y2l1}
y_1(|\alpha|/2) < 1 - y_2(|\alpha|/2) .
\end{equation}

\smallskip
Step 3. At last we may prove the theorem. 

\smallskip
(i) Suppose $\alpha>0$. Observe $Q(p)$ is strictly convex for $p \in (0,1/2]$, by \eqref{normalizedsecond}, \eqref{firstconvex} and \eqref{secondconvex}. It follows that the maximum of $Q(p)$ occurs either as $p \to 0$ or at $p=1/2$. As the rectangle degenerates, the limiting value is $\lim_{p \to 0} Q(p) = \alpha$, since $\lambda_2\big( (0,1-p); \alpha/2\big)$ converges to the finite value $\lambda_2\big( (0,1); \alpha/2\big)$ while $\lambda_1\big( (0,p); \alpha/2\big) \sim (\alpha/2)/(p/2)=\alpha/p$ as $p \to 0$ (using the blow-up rate from the proof of \autoref{1dimpos}). Meanwhile, the square has $Q(1/2) = \lambda_2\big( \mathcal{R}(1/2);\alpha/2\big)/4$. It follows from \eqref{biggeralpha1} and \eqref{biggeralpha2} that when $(0,\alpha_+)$ the maximum of $Q(p)$ occurs at $p=1/2$, and when $\alpha \in (\alpha_+,\infty)$ the maximum is achieved in the limit as $p \to 0$. 

In the borderline case $\alpha=\alpha_+$, equality holds in \eqref{biggeralpha1} and so $\lim_{p \to 0} Q(p) = Q(1/2)$, from which strict convexity of $Q$ implies $Q(p) < Q(1/2)$ for all $p \in (0,1/2)$. Therefore the square gives the largest value for $Q$. 

\smallskip
(ii) Suppose $\alpha=0$, in which case the second Neumann eigenvalue of the rectangle is $\pi^2/(1-p)^2$, remembering here that the long side has length $1-p$. Multiplying by the area $p(1-p)$ gives $\pi^2 p/(1-p)$, which for $p \in (0,1/2]$ is strictly maximal at $p=1/2$. In other words, the square maximizes the area-normalized second eigenvalue.  

\smallskip
(iii) Suppose $-4 \leq \alpha < 0$. Then $Q(p)$ is strictly increasing when $p \in (0,1/2]$, by using \eqref{normalizedsecond} and applying \eqref{specialalpha} with $\alpha/2$ instead of $\alpha$, and applying \eqref{specialalpha2} with $\alpha/2$ instead of $\alpha$ and $1-p$ instead of $p$. (The assumption $-4 \leq \alpha < 0$ ensures when using \eqref{specialalpha2} that $\min(1,2/|\alpha/2|)=1$.) Hence $Q(p)$ achieves its maximum at $p=1/2$ (the square). 

\smallskip
(iv) Suppose $-8 \leq \alpha < 4$, so that $\min(1,2/|\alpha/2|)=4/|\alpha|$. The argument in the preceding paragraph gives this time that $Q(p)$ is strictly increasing when $p \in (q(\alpha),1/2]$, where $q(\alpha)=1-4/|\alpha| \in (0,1/2]$. We will show $Q(p) < Q(q(\alpha))$ when $p \in (0,q(\alpha))$, so that once again $p=1/2$ gives the maximum of $Q$. To show $Q(p) < Q(q(\alpha))$, observe that $\lambda_1\big( (0,p); \alpha/2\big) p(1-p)$ is strictly increasing in $p$ by \eqref{specialalpha}, while the second eigenvalue $\lambda_2\big( (0,1-p); \alpha/2\big)$ equals zero at $p=q(\alpha)$ (by \autoref{secondeigen}, since $\alpha/2=-2/(1-q(\alpha))$) and is negative when $0 < p < q(\alpha)$ (by applying \eqref{specialalpha3}). 

\smallskip
(v) Suppose $-12 \leq \alpha < -8$. Let $c=|\alpha|/4 \leq 3$. From \eqref{specialalpha} with $\alpha/2$ in place of $\alpha$ we know $\lambda_1\big( (0,p); \alpha/2\big) p(1-p)$ is strictly convex for $p \in (0,1)$. From \eqref{specialalpha4} with $p$ replaced by $1-p$ we see that $\lambda_2\big( (0,1-p); \alpha/2\big) p(1-p)$ is strictly convex for $p \in (0,1-4/|\alpha|)$. This interval includes $(0,1/2]$ because $\alpha < -8$. Hence $Q(p)$ is strictly convex for $p \in (0,1/2]$, by \eqref{normalizedsecond}, and so the maximum of $Q$ occurs either as $p \to 0$ or at $p=1/2$. The limiting value as the rectangle degenerates is $\lim_{p \to 0} Q(p) = \alpha$ since $\lambda_1\big( (0,p); \alpha/2\big) \sim (\alpha/2)/(p/2) = \alpha/p$ as $p \to 0$ (using the blow-up rate from the proof of \autoref{1dimneg}). Thus when $\alpha \in [-12,\alpha_-)$ or $\alpha \in (\alpha_-,-8)$, the theorem follows from the comparison of the square and the degenerate rectangle in \eqref{biggeralpha1} and \eqref{biggeralpha2}. In the borderline case $\alpha=\alpha_-$, the square ($p=1/2$) gives the largest eigenvalue, by arguing as for the borderline case $\alpha=\alpha_+$ in part (i) above.

\smallskip
(vi) Suppose $\alpha< -12$. The normalized first eigenvalue $\lambda_1\big( (0,p); \alpha/2\big) p(1-p)$ is strictly convex for $p \in \big(y_1(|\alpha|/4),1\big)$, by \eqref{specialalpha7} with $\alpha/2$ in place of $\alpha$. Recall from \autoref{concavityH1} with $c=|\alpha|/4>3$ that the number $y_1(|\alpha|/4)$ lies between $0$ and $1/2$. Meanwhile, $\lambda_2\big( (0,1-p); \alpha/2\big) p(1-p)$ is strictly convex for $p \in (0,1/2]$, as observed above in part (v). Adding these two convex functions shows that $Q(p)$ is strictly convex for $p \in \big(y_1(|\alpha|/4),1/2\big]$. Thus $Q$ attains its maximum on that interval at one of the endpoints. 

On the remaining interval $\big(0,y_1(|\alpha|/4)\big)$, we will show $Q$ is strictly decreasing and hence attains its maximum at the left endpoint (as $p \to 0$). Armed with that fact, one completes the proof for $\alpha<-12$ by recalling from \eqref{biggeralpha2} that the function $Q(p)$ attains a bigger value as $p \to 0$ than it does at $p=1/2$. 

To show $Q$ is strictly decreasing on $\big(0,y_1(|\alpha|/4)\big)$, note $\lambda_1\big( (0,p); \alpha/2\big) p(1-p)$ is strictly decreasing for $p \in \big(0,y_1(|\alpha|/4)\big)$, by \eqref{specialalpha6}. Further, $\lambda_2\big( (0,1-p); \alpha/2\big) p(1-p)$ is strictly decreasing for $p \in \big(0,1-y_2(|\alpha|/4)\big)$, by replacing $\alpha$ with $\alpha/2$ and $p$ with $1-p$ in \eqref{specialalpha5}. That last interval contains $\big(0,y_1(|\alpha|/4)\big)$, due to \eqref{y1y2l1}, and so $Q(p)$ is strictly decreasing on $\big(0,y_1(|\alpha|/4)\big)$. 
\end{proof}

\begin{proof}[\bf Proof of \autoref{steklov}]
See the proof of \autoref{sigma1higherdim} for the relationship between the Steklov and Robin spectra. 

After rescaling the rectangle $\mathcal{R}$, we may suppose it has area $4$. Write $\mathcal{S}$ for the square of sidelength $2$ and hence area $4$ and perimeter $8$. \autoref{squareexample} gives that $\lambda_2(\mathcal{S};\alpha_0)=0$ where $\alpha_0 \simeq -0.68825$, and so $\sigma_1(\mathcal{S})= |\alpha_0|$. Thus the task is to prove $\sigma_1(\mathcal{R}) L(\mathcal{R}) \leq 8|\alpha_0|$, with equality if and only if the rectangle is a square. 

Choosing $\alpha= 8\alpha_0 \simeq -5.5$ in \autoref{lambda2L} yields that 
\begin{equation} \label{eq:squarebound}
\lambda_2\big( \mathcal{R};8\alpha_0/L(\mathcal{R}) \big) \leq \lambda_2\big(\mathcal{S};8\alpha_0/L(\mathcal{S})\big) = \lambda_2(\mathcal{S};\alpha_0) =0 .
\end{equation}
Also $\lambda_2(\mathcal{R};0)$ is positive. Since $\lambda_2(\mathcal{R};\alpha)$ is a continuous, strictly increasing function of $\alpha$, it follows that a unique number $\widetilde{\alpha} \in [8\alpha_0,0)$ exists for which $\lambda_2\big( \mathcal{R};\widetilde{\alpha}/L(\mathcal{R}) \big) = 0$. Hence $-\widetilde{\alpha}/L(\mathcal{R})=\sigma_1(\mathcal{R})$, and so $\sigma_1(\mathcal{R}) L(\mathcal{R}) = -\widetilde{\alpha} \leq 8|\alpha_0|$, as we needed to show. 

If equality holds then equality holds in \eqref{eq:squarebound}, and so the equality statement in \autoref{lambda2L}  implies $\mathcal{R}$ is a square. 
\end{proof}
\begin{proof}[\bf Proof of \autoref{gapD}]
\autoref{boxgap} shows the spectral gap for the box equals the spectral gap of its longest edge:
\[
(\lambda_2-\lambda_1)(\mathcal{B};\alpha) = (\lambda_2-\lambda_1)((0,s);\alpha)
\]
where we write $s$ for the length of the longest edge of the box. Since $s<D$ and the spectral gap of an interval is strictly decreasing as a function of the length (by \autoref{1dimpos}, \autoref{1dimzero}  and \autoref{1dimneg}), the conclusion of the theorem follows. 
\end{proof}
\begin{proof}[\bf Proof of \autoref{gapSV}]
Arguing as in the preceding proof, we see that to maximize the gap we must minimize the longest side $s$ of the box, subject to the constraint of fixed diameter. That is, we want to minimize the scale invariant ratio $s/D$ among boxes. The minimum is easily seen to occur for the cube, by fixing $s$ and increasing all the other side lengths to increase the diameter. 

The argument is similar under a surface area constraint since the scale invariant ratio $s^{n-1}/S$ is minimal among boxes for the cube, and under a volume constraint too since the ratio $s^n/V$ is minimal for the cube. 

\emph{Comment.} The version of the theorem with diameter constraint implies the one with volume constraint, since $s^n/V=(s/D)^n(D^n/V)$ and each ratio on the right is minimal at the cube. Similarly, the result with surface area constraint implies the one with volume constraint, since $s^n/V = (s^{n-1}/S)^{n/(n-1)}(S^{n/(n-1)}/V)$ and each ratio on the right is minimal at the cube. 
\end{proof}
\begin{proof}[\bf Proof of \autoref{ratio}]
The ratio can be rewritten in terms of the spectral gap as
\[
\frac{\lambda_2(\mathcal{B};\alpha)}{|\lambda_1(\mathcal{B};\alpha)|}
= 
\frac{\lambda_2(\mathcal{B};\alpha)-\lambda_1(\mathcal{B};\alpha)}{|\lambda_1(\mathcal{B};\alpha)|} + \sign(\alpha) .
\]
The numerator on the right is maximal for the cube having the same volume as $\mathcal{B}$, by \autoref{gapSV}, while the denominator is minimal for that cube by \autoref{lambda1higherdim}. Hence the ratio is maximal for the cube. 
\end{proof}
\begin{proof}[\bf Proof of \autoref{ratio2dim}]
In terms of the spectral gap, the ratio is
\[
\frac{\lambda_2(\mathcal{R};\alpha/L)}{\lambda_1(\mathcal{R};\alpha/L)}
= 
\frac{\lambda_2(\mathcal{R};\alpha/L)-\lambda_1(\mathcal{R};\alpha/L)}{\lambda_1(\mathcal{R};\alpha/L)A} A + 1 .
\]
The numerator on the right is maximal for the square having the same boundary length $L$ as $\mathcal{R}$, by the ``surface area'' version of \autoref{gapSV} applied with $\alpha/L$ instead of $\alpha$. And of course, the factor of $A$ is largest for the same square, by the isoperimetric inequality for rectangles. Meanwhile the denominator on the right side is positive (since $\alpha>0$) and is minimal for the square by \autoref{lambda1L}. Hence the right side is maximal for the square. 
\end{proof}
\begin{proof}[\bf Proof of \autoref{hearingdrum}]
After a rotation and translation, we may write the rectangle as $\mathcal{R} = \mathcal{I}(t) \times \mathcal{I}(s)$ where $t \geq s$. The first and second eigenvalues of this rectangle are the given information, and the task is to determine the side lengths $t$ and $s$. 

In terms of $t$ and $s$, the eigenvalues are   
\begin{align*}
\lambda_1(\mathcal{R};\alpha) & = \lambda_1(\mathcal{I}(t);\alpha) + \lambda_1(\mathcal{I}(s);\alpha) , \\
\lambda_2(\mathcal{R};\alpha) & = \lambda_2(\mathcal{I}(t);\alpha) + \lambda_1(\mathcal{I}(s);\alpha) , 
\end{align*}
by \autoref{firsteigenbox} and \autoref{secondeigenbox}. Subtracting, we obtain the spectral gap as 
\[
(\lambda_2-\lambda_1)(\mathcal{R};\alpha) = (\lambda_2-\lambda_1)(\mathcal{I}(t);\alpha) ,
\]
and so the value of the left side is also given information. The right side is the spectral gap of the interval $\mathcal{I}(t)$, which is a strictly decreasing function $t$ by \autoref{1dimpos} and \autoref{1dimneg}. Hence the longer sidelength $t$ of the rectangle is uniquely determined by the given information. 

The value of $\lambda_1(\mathcal{I}(s);\alpha)$ can then be determined from the formulas above. This eigenvalue depends strictly monotonically on the length $s$, by \autoref{1dimpos} and \autoref{1dimneg}, and hence the value of $s$ is uniquely determined. 

\emph{Comment.} The final step of the proof is where  the assumption $\alpha \neq 0$ is used --- when $\alpha=0$ the first eigenvalue is zero for every interval and hence is not strictly monotonic as a function of the length.
\end{proof}

\section*{Acknowledgments}
This research was supported by a grant from the Simons Foundation (\#429422 to Richard Laugesen) and travel support from the University of Illinois Scholars' Travel Fund. Conversations  with Dorin Bucur were particularly helpful, at the conference ``Results in Contemporary Mathematical Physics'' in honor of Rafael Benguria (Santiago, Chile, December 2018). I am grateful to Derek Kielty for carrying out numerical investigations in support of this research and pointing out relevant literature, and to Pedro Freitas for many informative conversations about Robin eigenvalues.

\end{document}